\documentclass[10pt,english]{amsart}

\usepackage{amsmath}
\usepackage{amsfonts}
\usepackage{amssymb}
\usepackage{graphicx}%
\usepackage{xcolor}
\providecommand{\U}[1]{\protect\rule{.1in}{.1in}}

\newtheorem{theorem}{Theorem}[section]

\newtheorem{corollary}[theorem]{Corollary}

\newtheorem{definition}[theorem]{Definition}

\newtheorem{lemma}[theorem]{Lemma}

\newtheorem{proposition}[theorem]{Proposition}
\newtheorem{remark}[theorem]{Remark}

\numberwithin{equation}{section}

\newcommand{\R}{\mathbf{R}}
\newcommand{\N}{\mathbf{N}}

\begin{document}

\author[Yong Liu]{Yong Liu}
\address{Department of Mathematics, University of Science and Technology of China, Hefei, China.}
\email{yliumath@ustc.edu.cn}

\author[Frank Pacard]{Frank Pacard}
\address{Centre de Mathématiques Laurent Schwartz, École Polytechnique, 91128 Palaiseau, France.}
\email{frank.pacard@polytechnique.edu}

\author[Juncheng Wei]{Juncheng Wei}
\address{Department of Mathematics, University of British Columbia, Vancouver, B.C. Canada, V6T 1Z2.}
\email{jcwei@math.ubc.ca}

\title[Bouncing Jacobi fields and the Allen-Cahn equation on surfaces]{Bouncing Jacobi Fields and the Allen-Cahn equation on surfaces}

\date{\today}

\maketitle

\begin{abstract}
    The Allen-Cahn functional is a well studied variational problem which appears in the modeling of phase transition phenomenon. This functional depends on a parameter $\varepsilon >0$ and is intimately related to the area functional as the parameter $\varepsilon$ tends to $0$. In the case where the ambient manifold is a compact surface, we give sufficient assumptions which guarantee the existence of countable families of critical points  of the Allen-Cahn functional whose nodal sets converge with multiplicity $2$ to a given embedded geodesic, while their energies and Morse indices stays uniformly bounded, as the parameter $\varepsilon$ tends to $0$. This result is specific to two dimensional surfaces and, for generic metric, it does not occur in higher dimension.
\end{abstract}

\section{Introduction}

Given a compact $m$-dimensional Riemannian manifold $(M,g)$ without boundary and a positive parameter $\varepsilon >0$, the Allen-Cahn equation arises as the Euler-Lagrange equation of the energy functional
\begin{equation}
\label{eq:1.1}
\mathcal E_{AC} (u) := \frac{\varepsilon}{2} \int_M |\nabla u|_g^2 \, \text{dvol}_g + \frac{1}{4\varepsilon}\, \int_M (1-u^2)^2\, \text{dvol}_g ,
\end{equation}
which models a phase transition phenomenon. A classical result of L. Modica \cite{Modica} asserts that, under natural assumptions, the Allen-Cahn functional $\Gamma$-converges to the volume functional.
This seminal result has triggered many studies to unravel the deep connections between the Allen-Cahn equation
\begin{equation}
\varepsilon^{2}\Delta_{g}u + u-u^{3} =0,
\label{eq:1.2}
\end{equation}
which is the Euler-Lagrange equation associated to the functional  (\ref{eq:1.1}) and minimal surfaces which are critical points of the volume functional.

\medskip

The zero set of solutions of (\ref{eq:1.2}) correspond to the \emph{transition layers} in the phase transition phenomenon and the $\Gamma$-convergence theory tells us that (under some natural assumption on their energy and Morse indices) the zero sets of sequences of solutions of the Allen-Cahn equation converge to minimal surfaces (eventually with multiplicity).

\medskip

In the past decade, we have seen important advances in the min-max theory of minimal surfaces in the varifold setting by C. Marques and A. Neves and their collaborators \cite{LMN, Marques, Mar0, Mar1}.
Given the strong links between the theory of minimal surfaces and the Allen-Cahn equation, these results have in turn triggered some work concerning the existence and behavior of sequences of critical points of the Allen-Cahn functional.

\medskip

The min-max theory for the Allen-Cahn equation which has been developed in \cite{Dey,GG,Gu} ensures the existence of min-max critical points of the Allen-Cahn functional, and, eventually, this provides a powerful tool to prove the existence of smooth minimal surfaces provided one is able to study the convergence of transition layers of the critical points as the parameter $\varepsilon$ tends to $0$.
In this context, a delicate question arises since the convergence of the transition layers to a minimal submanifold can be quite complicated. For example, transition layers could possibly converge to a minimal surface with multiplicity higher than $1$ or it might happen that the transition layers converge to some immersed minimal surface or even the union of minimal surfaces which intersect.

\medskip

In the case where the dimension of the ambient manifold is equal to $3$, O. Chodosh and C. Mantoulidis \cite{Man1} have proven that, for generic metrics on the manifold,  uniform bounds on the Morse index and the energy of solutions of (\ref{eq:1.2}) prevent the convergence of the transition layers with multiplicity to a minimal surface.

\medskip

Recall that the Jacobi operator about a given minimal surface $N$ is defined by
\[
\mathfrak J_N := - \left( \Delta_N + |A_N|^2+ \text{Ric}(\nu, \nu)\right) ,
\]
where $|A_N|^2$ is the square of the norm of the second fundamental form about $N$ and $\nu$ a unit normal vector field along $N$. The following result is proven in \cite{Man1}:
\begin{theorem}
\label{th:1.1}
\cite{Man1}
Let $N$ denote a minimal surface which is embedded in a $3$-dimensional manifold $M$ and is constructed as the limit, as $\varepsilon$ tends to $0$, of zero sets of solutions to the Allen–Cahn equation with uniformly bounded indices and energies.
If the convergence of the transition layers towards $N$ occurs with multiplicity or is two-sided, then $N$ is stable and carries a positive Jacobi field (on its two-sided double cover, in the second case).
\end{theorem}
Recall that, for bumpy metrics \cite{Whi91,Whi17}, the Jacobi operator does not carry any Jacobi field.
This implies that, for generic metrics on the ambient $3$-manifold, if an embedded minimal surface $N$ is obtained  as in Theorem~\ref{th:1.1}, then it is two-sided and the convergence of the nodal sets of the solutions of the Allen-Cahn equation towards $N$, does not occur with multiplicity larger than one.

\medskip

One of the motivation of the present paper is to understand the set of solutions of the Allen-Cahn equation which have energy and indices uniformly bounded as the parameter $\varepsilon$ tends to $0$.
The \emph{Morse index} of a solution $u$ of (\ref{eq:1.2}) is defined to be the dimension of the space of (smooth) functions over which the quadratic form
\[
{\mathcal Q}_{u}(w) := \varepsilon \, \int_M |\nabla w|_g^2 \, \text{dvol}_g +\frac{1}{\varepsilon} \, \int_M   \, (3u^2-1) \, w^2 \, \text{dvol}_g,
\]
is negative definite.

\medskip

In view of all these developments, it is natural to ask whether Theorem~\ref{th:1.1} also holds when the ambient manifold is a $2$-dimensional surface. In other words, is it possible to construct sequences of solutions of the Allen-Cahn equation which are defined on a compact surface without boundary, whose energies and Morse indices are uniformly bounded and whose transition layers converge with multiplicity larger than one, to a given embedded geodesics as the parameter $\varepsilon$ tends to $0$?

\medskip

This question will be answered affirmatively in the present paper.

\medskip

To better understand our result, let us recall a couple of constructions which are relevant to the present paper. Given an embedded minimal hypersurface $N$ in $M$ a $m$-dimensional compact manifold without boundary, solutions whose zero set is a (small) normal graph over $N$ have been constructed using a Lyapunov-Schmidt reduction argument in \cite{Pacard, Pac-Role}, provided the Jacobi operator about $N$ does not carry any Jacobi field.
This approach has been further developed in \cite{Manuel0} to construct counter-examples to the De Giorgi conjecture for the Allen-Cahn equation in $\R^{9}$ and entire solutions of the Allen-Cahn equation whose zero set is close to a given embedded finite total curvature minimal surface in $\R^{3}$ \cite{M0}.
In all these papers, the transition layer (i.e. the zero set) of the solutions is a (small) normal graphs over a minimal hypersurface $N$.
In particular, they are diffeomorphic to $N$ and hence this corresponds to the case where the zero set of the solutions to the Allen-Cahn equation converges to a minimal hypersurface with multiplicity one.
The existence of solutions to the Allen-Cahn equation whose transition layers consist of several disjoint (small) graphs over a given minimal surface has been proven in \cite{Manuel2}, provided the potential $|A_N|^2 +\text{Ric} (\nu, \nu)$ which appears in the Jacobi operator, is positive along $N$.
In this case, the position of the transition layers of the solutions of the Allen-Cahn equation is governed by solutions of the the so-called \emph{Jacobi-Toda system} which will also play an important role in the present paper.
This last result provides examples of solutions to the Allen-Cahn equation whose transition layers converge to a given embedded minimal surface with multiplicity larger than one, as $\varepsilon$ tends to $0$.
However, these solutions do not satisfy all assumptions of the Theorem of O. Chodosh and C. Mantoulidis, since their Morse indices are of the order $\left\vert\ln\varepsilon\right\vert$, and hence tend to infinity as $\varepsilon$ tends to $0$.

\medskip

In the present paper, we concentrate our attention to the case where the ambient manifold $(M, g)$ is a compact surface without boundary and $\gamma$ is a closed geodesic embedded in $M$.
In this case, the Jacobi operator about $\gamma$ reduces to
\[
\mathfrak J_\gamma    := -( \partial_s^2 + K_g),
\]
where $s$ is the arc-length along $\gamma$ and where $K_g$ denotes the Gauss curvature of $(M,g)$ along $\gamma$. Let us recall the well known:
\begin{definition}
A closed geodesic $\gamma$ is said to be non-degenerate if the Jacobi operator $\mathfrak J_\gamma$ does not have any nonzero Jacobi field, i.e. it has trivial kernel.
\end{definition}
We also define $\text{Ind}\, (\gamma)$, the \emph{Morse index} of a closed geodesic $\gamma$, to be equal to the maximal dimension of the space of (smooth) functions defined on $\gamma$, over which the quadratic form
\[
Q_\gamma(v) := \int_\gamma \left( |\partial_s v|^2 - K_g \, v^2 \right) \, ds,
\]
is negative definite.

\medskip

To state our result, we need to introduce a new notion which we have christened \emph{Bouncing Jacobi Field}. It will be convenient to assume that the geodesic $\gamma$ we are given, is parameterized by arc-length and, to allay notations, we will identify $\gamma$ with $\R/\hspace{-1mm}\sim$ where $s\sim s'$ if and only if $s'-s\in |\gamma|\, {\bf Z}$, where $|\gamma|$ is the length of $\gamma$. Therefore, $s$ will denote both a point in $\gamma$ and a point in $\R/\hspace{-1mm}\sim$.

\medskip

Finally, if $s_1, \ldots, s_n$ are point of $\gamma$, we will assume that they (or their representatives) are arranged so that $0 \leq s_1 <  \ldots  < s_n < |\gamma|$ and we will agree that $s_0 := s_n -|\gamma|$ and $s_{n+1} = s_1 + |\gamma|$.

\begin{definition}[Bouncing Jacobi Field]
\label{de:1.3}
We will say that a real valued function $\Phi : \gamma \to \R$ is a \emph{Bouncing Jacobi Field} with $n\geq 1$ minimums if there exist $n$ distinct points $s_1,\ldots,s_n\in \gamma$, such that $\Phi$ is continuous on $\gamma$, smooth on $\gamma-\left\{  s_1,\ldots, s_n\right\}$ and if it satisfies
\[
\left\{
\begin{array}{rlllllll}
\mathfrak J_\gamma \Phi = 0 & \text{and} & \Phi  >   1 \quad \text{on}\quad \gamma -\{s_1, \ldots, s_n\},\\[3mm]
\Phi (s_j) = 1 & \text{and} &
\partial_s \Phi (s_j^+) = - \partial_s \Phi  (s_j^-) >  0 \quad \text{for all} \quad j =1, \ldots, n.
\end{array}
\right.
\]
\end{definition}

In this definition as well as in the rest of the paper, if $f : \gamma-\{s_1, \ldots, s_n\}\to \R$ is a function, then $f(s^+_j)$ and $f(s^-_j)$ denote respectively the left and right limits of $f$ at $s_j$, assuming these limits are defined.

\medskip

We will prove that the points $s_1, \ldots, s_n$ where a given Bouncing Jacobi Field is minimal are critical points of the function
\[
\mathcal H_n (s_1, \ldots, s_n) := \sum_{j=1}^n  \left( \inf_{\phi-1\in H_{0}^{1}\left(s_i,s_{i+1} \right)} \int_{(s_i, s_{i+1})}\left(|\partial_s\phi|^2-K_g \, \phi^{2} \right) \, ds \right).
\]
The index of a Bouncing Jacobi Field $\Phi$ will be denoted by $\text{Ind}\, (\Phi)$.
It is by definition equal to the number of negative eigenvalues of the symmetric matrix associated to the Hessian of $\mathcal H_n$ at $(s_1, \ldots, s_n)$ where $\Phi$ is minimal.
The nullity of the Bouncing Jacobi Field $\Phi$ is by definition equal to the dimension of the null-space of the matrix associated to the Hessian of $\mathcal H_n$ at $(s_1, \ldots, s_n)$. A Bouncing Jacobi Field is said to be non-degenerate if its nullity is $0$.

\medskip

The existence of Bouncing Jacobi Fields is guaranteed by the following result:
\begin{theorem}
\label{th:1.4} Let $\gamma$ be an embedded  geodesic in $M$ and $n \geq 1$. Assume that the Gauss curvature $K_g$ is positive along $\gamma$ and
\[
|\gamma|^2 \,  \sup_{\gamma} K_g < \pi^2\, n^2.
\]
Then, there exists a \emph{Bouncing Jacobi Field} which has $n$ minimums and for which, the sum of the nullity and index is equal to $n$.
\end{theorem}

We will also prove that, if $K_g$ is positive and $n\geq 1$, then
\[
|\gamma|^2 \,  \inf_{\gamma} K_g \leq \pi^2\, n^2,
\]
is a necessary condition for the existence of a Bouncing-Jacobi Field with $n$ minimum along $\gamma$.
As we will see (Corollary~\ref{co:5.6}),
\[
2n \geq \text{Ind} (\gamma).
\]
is also a necessary condition for a Bouncing Jacobi Field $\Phi$ with $n$ minimums to exist.

\medskip

Bouncing Jacobi Fields are at the heart of our construction. Starting from a Bouncing Jacobi Field with suitable non-degeneracy condition, we can construct solutions of Allen-Cahn equation with bounded energy and indices.

\medskip

Let ${\bf G}^1$ and ${\bf G}^2$ be two (possibly empty) finite collections of closed, embedded, non-degenerate geodesics on $(M,g)$. We assume that geodesics in ${\bf G}^1\cup \bf G^2$ are disjoint and we further assume that, for all $\gamma \in {\bf G}^2$, the Jacobi operator $\mathfrak J_{\gamma}$ carries a non-degenerate \emph{Bouncing Jacobi Field} $\Phi= \Phi_\gamma$ with $n=n_{\gamma} \geq 1$ minimums. Finally, assume that ${\bf G}^1$ is the boundary of a smooth compact domain(This is to guarantee that an approximate solution can be defined on the manifold) and that each geodesic in ${\bf G}^2$ is two-sided (namely, that its normal bundle is orientable).

\medskip

The main result of the paper is the following:
\begin{theorem}
\label{th:1.5}
Under the above assumptions, for all $\varepsilon$ small enough, the Allen-Cahn equation has a solution $u_{\varepsilon}$ whose Morse index is equal to
\[
\sum_{\gamma \in {\bf G}^1} \text{\rm Ind} \, (\gamma)+\sum_{\gamma \in {\bf G}^2} (\text{\rm Ind} \, (\gamma)+ n_\gamma + \text{\rm Ind} (\Phi_\gamma) ),
\]
and whose energy converges to
\[
\frac{2\sqrt 2}{3} \, \left(
\sum_{\gamma \in {\bf G}^1}  |\gamma|+2\sum_{\gamma \in {\bf G}^2}  |\gamma|\right ),
\]
as $\varepsilon$ tends to $0$. Moreover, the transition layers of $u_{\varepsilon}$ converge to geodesics in ${\bf G}^2$ with multiplicity two and converge to geodesics in ${\bf G}^1$ with multiplicity one.
\end{theorem}
The constant $\frac{2\sqrt 2}{3}$ which appears in the statement of the theorem is nothing but the energy of the one dimensional heteroclinic solution
\[
x \mapsto \tanh \left(\frac{x}{\sqrt 2}\right).
\]

Let us explain how this result relates to former results on the construction of solutions of the Allen-Cahn equation.

\medskip

First, the above result is a combination of the results which will be proved in the present paper and the result of F. Pacard and M. Ritoré \cite{Pacard, Pac-Role} which addresses the case where ${\bf G}^2$ is empty.

\medskip

Our result is, in some sense, a generalization of earlier results  obtained in \cite{Manuel1} and \cite{Mich1}.
In these two papers, the Allen-Cahn equation is considered over a smooth bounded domain in the Euclidean plane and solutions are assumed to have $0$ Neumann condition on the boundary of the domain.
In this context, a geodesic is a line segment $\gamma$ meeting the boundary of the domain orthogonally.
The result in \cite{Mich1} corresponds to the case where the transition layer of the solutions converges to $\gamma$ with multiplicity $1$.
The result in \cite{Manuel1} corresponds exactly to solutions of the Allen-Cahn equation whose transition layers converge to $\gamma$ with multiplicity $2$ and in fact, the solutions in this paper are constructed starting from a Bouncing Jacobi Field, which in this case is continuous, piecewise  affine function of the form $s \mapsto 1 + a \, |s- \bar s|$, satisfying some boundary condition.

\medskip

Observe that, in Theorem~\ref{th:1.5}, we do not assume that the Gauss curvature is everywhere positive along each geodesic of ${\bf G}^2$.
If one assumes that the Gauss curvature is positive along a closed embedded geodesic, then the result in \cite{Manuel2} already provides solutions of the Allen-Cahn equation whose transition layers converges to $\gamma$ with multiplicity $2$ as the parameter $\varepsilon$ tends to $0$.
But as already explained, the Morse indices of these solutions tends to $+\infty$ as $\varepsilon$ tends to $0$, while the solutions provided by Theorem~\ref{th:1.5} have Morse indices which does not depend on $\varepsilon$.

\begin{remark}
\label{re:1.6}
In Theorem~\ref{th:1.5}, we ask that the normal bundle over each geodesic of ${\bf G}^2$ is two-sided.
This assumption can be relaxed if one requires that both the geodesic and the Bouncing Jacobi Fields, extended to the two-sided double-cover, are non-degenerate.
\end{remark}

\medskip

Our result requires the geodesics and the Bouncing Jacobi Fields to be non-degenerate. These two non-degeneracy conditions do hold for a generic choice of the ambient metric on $M$.
A metric $g$ on a surface $(M,g)$ is said to be a \emph{bumpy metric} if there is no immersed  closed geodesic carrying a non-trivial Jacobi field.
Recall that B. White \cite{Whi91, Whi17} has proven that bumpy metrics are generic in the sense of Baire category. The existence of non-degenerate geodesics carrying non-degenerate Bouncing Jacobi Fields is ensured by the Theorem~\ref{th:1.4} and the:
\begin{theorem}
\label{th:1.7}
For a generic choice of a $\mathcal C^k$, with $k\geq 3$, metric $g$ on $M$ (in the sense of Baire Category), embedded, closed  geodesic $\gamma$ in $(M,g)$ and all \emph{Bouncing Jacobi Field} they carry are non-degenerate.
\end{theorem}

\medskip

Collecting Theorem~\ref{th:1.7}, Theorem~\ref{th:1.5} and Theorem~\ref{th:1.4}, we get the:
\begin{corollary}
\label{co:1.8}
For generic metric on $M$, if $\gamma$ is an embedded closed geodesic such that $K_g >0$ along $\gamma$, and if $n\in \N$ satisfies
\[
n^2 \, \pi^2 > |\gamma|^2 \,  \sup_{\gamma}K_g,
\]
then, for all $\varepsilon >0$ small enough, there exist solutions of the Allen-Cahn equation whose Morse index is given by $\text{Ind} \, (\gamma) + 2n$ and whose transition layers converge to $\gamma$ with multiplicity $2$.
\end{corollary}

Given a closed embedded geodesic in $M$ and $n\geq 1$, uniqueness of the Bouncing Jacobi Fields having $n$ minimums is not true in general.
For example, when $n=1$ and under suitable assumption on the Gauss curvature, we will see that there are at least $2$ distinct Bouncing Jacobi Fields with one minimum, giving rise to two distinct solutions of the Allen-Cahn equation, one of which having a Morse index strictly smaller than the one described in Theorem~\ref{th:1.4}.
Regarding other construction of solutions with bounded Morse indices, it is also worth mentioning that existence and symmetry of ground states for the Allen-Cahn equation on the round sphere have been studied in \cite{Caju}.

\medskip

The tools developed to prove Theorem~\ref{th:1.5} can  very likely be generalized to produce solutions of the Allen-Cahn equations whose transition layers converge to a given geodesic with multiplicity $k\geq 3$, while their Morse indices stays constant and their energy remains uniformly bounded as $\varepsilon$ tends to $0$.

\medskip

As we have discussed above, Bouncing Jacobi Fields do exist in abundance, at least when the Gauss curvature along a geodesic is positive, while, when the Gauss curvature along the geodesic is negative, Bouncing Jacobi Fields do not exist and we expect that, solutions of the Allen-Cahn equation whose transition layers converge to a given geodesic $\gamma$ with  multiplicity larger than $1$ do not exist when the Gauss curvature along $\gamma$ is negative.

\medskip

To complete the picture, let us mention that, in a forthcoming paper \cite{LPW2}, we will consider the case where the transition layers of the solutions of the Allen-Cahn equation converges to the union of of geodesics that are not necessarily disjoint.
In which case, we obtain the existence of solutions to the Allen-Cahn equation whose Morse index is equal to the sum of the Morse indices of the geodesics and the number of intersection points and self intersections of the union of the geodesics.
These results are first steps toward a classification of critical points of the Allen-Cahn functional with controlled energy and Morse index.

\medskip

The paper is organized in the following way.
In Section 2, we discuss the existence and non existence of Bouncing Jacobi Fields which have been defined in Definition~\ref{de:1.3} and, in Section 3, we study the Morse Index of a given Bouncing Jacobi Field.
Bouncing Jacobi Fields are at the heart of our construction since they are an essential element which will allow us to prove, in Section 4, the existence of special solutions $\Psi : \gamma \to \R$ of the Jacobi-Toda equation
\[
\varepsilon^2 \, \mathfrak J_\gamma \Psi + \bar c^2\, e^{-2 \Psi} =0,
\]
provided $\varepsilon >0$ is small enough, where $\bar c >0$ is a universal constant. These special solutions are obtained by smoothing the singularities (i.e. the points  where the function is continuous and not smooth) of a given Bouncing Jacobi Field $\Phi$ using suitably scaled copies of solutions $T : \R \to \R$ of the Toda equation
\[
\partial_s^2 T - e^{-2T} =0.
\]
Section 5 will be devoted to the study of the Morse index of the solutions $\Psi$ constructed in the previous section.
We will prove that the Morse index of $\Psi$ is the sum of the Morse index of $\gamma$ and the number of minimums of $\Phi$.
In Section 6, starting from $\Psi$, a special solutions of the Jacobi-Toda equation, we will construct a solution of the Allen-Cahn equation whose transition layers are normal graphs over $\gamma$ for the functions $\pm \varepsilon\, \Psi$.
This construction  is by now standard and starts with the construction of an approximate solution  which is then perturbed into a genuine solution of the Allen-Cahn equation.
Once the existence of the desired solution of the Allen-Cahn equation is complete, we will study their Morse index in Section 7.
Therefore, in some sense, we reduce the existence of solutions to the Allen-Cahn equation whose transition layers converge to a given geodesic with multiplicity $2$ to the existence of Bouncing Jacobi Fields.
Finally, Section 8 is devoted to the proof of the fact that, for generic metrics on $M$, Bouncing Jacobi Fields are non-degenerate.

\medskip

\textbf{Acknowledgement} Y. Liu is supported by  the National Key R\&D Program of China 2022YFA1005400 and NSFC 11971026, NSFC 12141105. J. Wei is partially supported by NSERC of Canada.
\section{Bouncing Jacobi Fields}
In the case where the Gauss curvature is constant (positive) along the embedded geodesic $\gamma$, the existence of Bouncing Jacobi Fields with $n$ minimums is straightforward. Indeed, we have the:
\begin{lemma}
\label{le:2.1}
When $K_g$ is constant and positive along $\gamma$, Bouncing Jacobi Fields with $n$ minimums do exist if and only if $\pi^2 \, n^2 > |\gamma|^2 \, K_g$, where $|\gamma|$ denotes the length of $\gamma$.
\end{lemma}
\begin{proof}
We set $s_n:= \frac{|\gamma|}{2n}$, the Bouncing Jacobi field (unique up to translation) is given by function
\[
\Phi (s)= \frac{\cos \left( \sqrt{K_g} s\right) \, }{\cos\left( \sqrt{K_g} s_n\right)},
\]
on  $\left(-s_n, s_n\right)$ and extended to be $2s_n$-periodic.
\end{proof}

In the general case, we have the following theorem, where the existence part is a direct application of Theorem 1.2 in \cite{Torr} (the equation being known as Hill's equation with obstacle).

\begin{theorem}
\label{th:2.2} Assume that $K_g \geq  0$ along $\gamma$ and let $n \geq 1$ be given. Then, $\mathfrak J_\gamma$ has at least one Bouncing Jacobi Field with $n$ local minimums provided
\begin{equation}
\pi^2 \, n^2 > |\gamma|^2 \, \sup K_g,
\label{eq:2.1}%
\end{equation}
and $\mathfrak J_\gamma$ does not have any Bouncing Jacobi Field with $n$ local minimums when
\begin{equation}
\label{eq:2.2}
\pi^2 \, n^2 < |\gamma|^2 \, \inf K_g.
\end{equation}
\end{theorem}

The method used in \cite{Torr} is  the Poincaré-Birkhoff Theorem about the existence of fixed points for area preserving maps in dimension two.
We also refer to  Theorem 8.4 and Corollary 8.6 in \cite{Muc}, for a detailed introduction to the Poincaré-Birkhoff Theorem.

\medskip

We point out that the proof in \cite{Torr} does not provide any information about the Morse indices of the Bouncing Jacobi Fields.
For this reason, we give here a (rather independent) proof of Theorem~\ref{th:2.2}.
As we will see, our proof captures more information about the Morse indices of the solutions.

\medskip

Our proof is based on the existence of a max-min solution for some energy we now describe.
For $s_{1} < s_{2}$, we define
\[
H\left(s_{1},s_{2}\right) :=\inf_{\phi-1\in H_{0}^{1}\left(s_{1},s_{2}\right)} \int_{s_{1}}^{s_{2}}\left(|\partial_s\phi|^2-K_g \, \phi^{2} \right) \, ds .
\]
Observe that the infimum may well be equal to $-\infty$. On the other hand, if $H\left(s_{1},s_{2}\right) > -\infty,$ then the infimum is achieved for some function $\phi$ solution of
\[
\mathfrak J_\gamma \phi =0,
\]
in $(s_1, s_2)$ with boundary data $\phi (s_1) = \phi (s_2)=1$.
Standard arguments imply that $\phi > 0$ in $[s_1, s_2]$ (since $|\phi|$ is also a minimizer).
The maximum principle, together with the fact that we have assumed that $K_g$ is positive along $\gamma$, then implies that
\[
\phi \left(s\right) > 1,
\]
in $(s_1, s_2)$.
Finally, using an integration by parts, we conclude that
\[
H\left( s_{1},s_{2}\right) = \partial_s \phi \left(  s_{2}\right) -\partial_s \phi \left(s_{1}\right) \leq  0.
\]

This being understood, we consider the set
\[
\Omega_n :=\left\{ \left(s_{1},s_2, \ldots, s_{n}\right) \, : \, 0\leq s_{1}<  \ldots < s_{n} < |\gamma| \right\},
\]
and, for all $\left(s_{1},\ldots , s_{n}\right) \in \Omega_n$, we define
\[
\mathcal H_n\left( s_{1},\ldots ,s_{n}\right) := \sum_{i=1}^{n} H\left(s_{i},s_{i+1}\right)  .
\]
where we agree that $s_{n+1}:=s_{1}+|\gamma|$.

\medskip

\begin{proof}
[Proof of Theorem~\ref{th:2.2}] The proof is decomposed into three parts.
In the first part, we give a sufficient condition for the existence of a critical point of $\mathcal H_n$.
Then, we derive the first variation of $\mathcal H_n$ and prove that  a critical point  of this functional is associated to a Bouncing Jacobi Field with $n$ minimums.
Finally, we give some necessary conditions for the existence of Bouncing Jacobi Field with $n$ minimums.

\begin{lemma}
\label{le:2.3}
Assume that $K_g > 0$ and
\begin{equation}
\pi^2\,n^2>|\gamma|^2 \displaystyle \sup_\gamma K_g.
\label{eq:2.3}
\end{equation}
Then, there exists a critical point of $\mathcal H_n$ and in fact this point can be chosen to be a point where $\mathcal H_n$ is maximal.
\end{lemma}
\begin{proof}
As we have mentioned above, since it is assumed that $K_g > 0$ along $\gamma$, there holds
\[
H\left(s_{i},s_{i+1}\right) \leq  0,
\]
and hence $\mathcal H_n\leq 0$.
To produce critical points of $\mathcal H_n$, it is therefore reasonable to look for a point of $\Omega_n$ where $\mathcal H_n$ achieves a local maximum.

\medskip

Let us prove that (\ref{eq:2.3}) ensures that $\displaystyle \sup_{\Omega_n}\mathcal H_n > -\infty$. Indeed, the $n$-tuple $\left(\bar{s}_{1}, \ldots ,\bar {s}_{n}\right)$ where $\bar{s}_{j} := \frac{|\gamma|}{n}\,j$ belongs to $\Omega_n$ and, using Poincaré inequality, we estimate
\begin{align*}
H\left(\bar{s}_{i},\bar{s}_{i+1}\right) & =\inf_{\phi\in H_{0}^{1}\left(\bar{s}_{i},\bar{s}_{i+1}\right)  }\int_{\bar{s}_{i}}^{\bar{s}_{i+1}}\left(|\partial_s \phi|^2-K_g \, \left(\phi+1\right)^{2}\right) ds\\
& \geq\inf_{\phi\in H_{0}^{1}\left(  \bar{s}_{i},\bar{s}_{i+1}\right)} \int_{\bar{s}_{i}}^{\bar{s}_{i+1}}\left( \left(  \frac{n^{2}\pi^2}{|\gamma|^2}-K_g\right) \phi^{2}-2K_g\phi- K_g\right) ds \\
&  > -\infty.
\end{align*}
Therefore,
\[
\mathcal{M}_n:=\sup_{{\bf s}\in\Omega_n} \mathcal{H}_n\left( {\bf s}\right)  > -\infty .
\]

We claim that $\mathcal{M}_n$ is achieved in $\Omega_n$. The \emph{key observation} is the fact that, if $s < t$ satisfy $H\left(s,t\right)  >-\infty,$ then
\[
H\left(s,t\right) < H\left(s,\frac{s+t}{2}\right) + H\left(  \frac{s+t}{2},t\right).
\]
Therefore, introducing an extra minimum increases the energy.

\medskip

Now, let ${\bf s}_i \in \Omega_n$ be a maximizing sequence for $\mathcal H_n$ over $\Omega_n$. Then, by definition,
\[
-\infty < \mathcal M_n = \lim_{i\to +\infty} \mathcal H_n({\bf s}_i).
\]
and, up to a subsequence, we can assume that $({\bf s}_i)_{i \in \bf N}$ converges to ${\bf s}_\infty \in \overline \Omega_n$.
We now argue by contradiction and assume that ${\bf s}_\infty \notin \Omega_n$ and hence ${\bf s}_\infty \in \Omega_k \cap \partial \Omega_n$, for some $k < n$.  It is easy to see that
\[
-\infty < \mathcal M_n = \lim_{i\to +\infty} \mathcal H_n ({\bf z}_i) = \mathcal H_k ({\bf s}_\infty).
\]
But, using the above key observation $n-k$ times, if this were the case, we would conclude that
\[
\mathcal H_k ({\bf z}_\infty) < \mathcal M_{n},
\]
which is a contradiction and this proves the claim. To conclude, we have obtained a critical point of $\mathcal H_n$ and, in fact, this critical point maximises $\mathcal H_n$ over $\Omega_n$.
\end{proof}

\medskip

\begin{remark}
\label{re:2.4}
In the case where $n=1,$ the analysis simplifies since we just consider the function
\[
\mathcal{H}_1\left( s\right) :=H\left(  s,s+|\gamma|\right),
\]
for $s \in \R$.
This is a continuous and periodic function (of period $|\gamma|$), and if we assume that $|\gamma|^2 \, \sup_\gamma K_g < \pi^2$, the function $\mathcal{H}_1$ is bounded from below.
It follows that $\mathcal{H}_1$ has at least two critical points, corresponding to a global maximum and minimum of $\mathcal{H}_1$ on $\gamma$.
If the maximum equals minimum, then $\mathcal{H}_1$ is constant and hence has infinitely many critical points (this also implies that $K_g$ is constant along $\gamma$).
\end{remark}

\medskip

We will now prove that the function $\Phi$ associated to a \textit{critical point} of $\mathcal H_n$ is a Bouncing Jacobi Fields with $n$ minimums.
This will be a consequence of the computation of the first variation of $\mathcal H_n$ at a point $\left(s_{1},\ldots ,s_{n}\right) \in \Omega_n$.

\medskip

\begin{proposition}
Assume that $K_g >0$ along $\gamma$ and that $(s_1, \ldots, s_n)$ is a critical point of $\mathcal H_n$, then the associated function $\Phi$ is a Bouncing Jacobi Field with $n$ minimums.
\label{pr:2.5}
\end{proposition}
\begin{proof}
Let us assume that $(s_1, \ldots, s_n)$, with $s_1< \ldots < s_n$ is a critical point of $\mathcal H_n$.
We agree that $s_{n+1}:=s_1+|\gamma|$ and $s_0: = s_n -|\gamma|$. We denote by $\Phi$ the function associated to this critical point and denote the restriction of $\Phi$ to $(s_j, s_{j+1})$ by $\phi_j$.
Therefore, for all $j=1, \ldots, n$, the function $\phi_j$ is a solution of $\mathfrak J_\gamma \phi_j =0$ in $(s_j, s_{j+1})$ with $\phi_j(s_j)= \phi_j(s_{j+1})=1$, and $\phi_j >1$ on $(s_j, s_{j+1})$ (observe that, since we have assumed that $K_g >0$, $\phi_j$ can't have a positive local minimum in $(s_{j}, s_{j+1})$).

\medskip

For all $s^*_j$ close to $s_j$ (for $j=1, \ldots, n$), we would like to understand the solution of $\mathfrak J_\gamma \tilde \phi_j =0$ on $(s_j^*, s_{j+1}^*)$ which satisfies $\tilde \phi_j (s_j) = \tilde \phi_j (s_{j+1})=1$ and $\tilde \phi_j >0$ in $(s^*_j, s^*_{j+1})$ and which is close to $\phi_j$.

\medskip

Given $\delta_j^-, \delta_j^+ \in \R$, for $j=1, \ldots , n$, we can solve
\[
\mathfrak J_\gamma     \eta_j = 0,
\]
in $(s_j, s_{j+1})$ with boundary data $\eta_j (s_j)= \delta_j^+$ and $\eta_j (s_{j+1}) = \delta_{j+1}^-$.
The existence of $\eta_j$ is guarantied by the fact that $\phi_j > 0$ over $[s_j, s_{j+1}]$ and hence the equation $\mathfrak J_\gamma     \zeta = 0$ on $(s_j, s_{j+1})$ with boundary data $\zeta (s_j) = \zeta (s_{j+1})=0$ does not have any nontrivial solution.

\medskip

Since we are dealing with a linear second order ordinary differential equation with bounded coefficients, the functions $s \mapsto (\phi_j + \eta_j) (s)$ which are smooth on $[s_j, s_{j+1}]$ can be extended as  \emph{smooth} functions $\psi_j$, solutions of $\mathfrak J_\gamma     \psi_j=0$ in an open interval containing $[s_j, s_{j+1}]$.
Using the fact that that $\phi_j (s_j)=1$ and also that
\[
\partial_s \psi_j (s_j) = \partial_s (\phi_j+\eta_j) (s_j) \qquad \text{and} \qquad
\partial_s \psi_j (s_{j+1}) = \partial_s (\phi_j+\eta_j) (s_{j+1}),
\]
we get the Taylor expansion of $\psi_j$ at $s_j$
\[
\begin{array}{rllll}
\psi_j (s_j+s)  & = &  \displaystyle  1 + \delta^+_j +  \partial_s (\phi_j +\eta_j) (s_j) \, s - K_g(s_j) (1+\delta^+_j) \, \frac{s^2}{2} \\[3mm]
 & - & \displaystyle \left(\partial_s K_g(s_j) \, (1+\delta^+_j) + K_g(s_j)\, \partial_s (\phi_j+\eta_j)(s_j)\right) \, \frac{s^3}{6} + \mathcal O (s^4),
\end{array}
\]
and the Taylor expansion of $\psi_{j-1}$ at $s_{j}$
\[
\begin{array}{rllll}
\psi_{j-1} (s_j+s) & = & \displaystyle 1 + \delta^-_{j} + \partial_s (\phi_{j-1} +\eta_{j-1}) (s_j) \, s - K_g(s_j) (1+\delta^-_j) \, \frac{s^2}{2} \\[3mm]
& - & \displaystyle \left(\partial_s K_g(s_j) (1+\delta^-_j) + K_g(s_j) \, \partial_s (\phi_{j-1}+\eta_{j-1})(s_j)\right) \, \frac{s^3}{6} + \mathcal O (s^4).
\end{array}
\]
Now, $s_j^*$ being chosen close enough to $s_j$, for all $j=1, \ldots, n$, we can determine  $\delta_j^\pm$ (close to $0$) so that
\[
\psi_{j-1} (s_j^*) = \psi_j  (s_j^*) = 1.
\]
It is easy to check that the  $\delta^\pm_j$ depend smoothly on the $s^*_j$ and we have the expansions \begin{equation}
\delta_j^+ - \delta_j^- = - N_j \, \dot s_j +\mathcal O (\dot s^2),
\label{eq:2.4}
\end{equation}
and
\begin{equation}
\begin{array}{rllll}
\delta^+_j + \delta^-_j & =& \displaystyle  - \partial_s (\phi_j + \phi_{j-1} )(s_j) \, \dot s_j - (\partial_s\eta_j + \partial_s\eta_{j-1} )(s_j) \, \dot s_j \\[3mm]
& + & \displaystyle K_g(s_j) \, \dot s_j^2 + \mathcal O(\dot s^3),
\end{array}
\label{eq:2.5}
\end{equation}
where we have defined
\[
\dot s_j := s_j^* -s_j,
\]
and
\[
N_j := -\partial_s (\phi_j - \phi_{j-1}) (s_j).
\]
It is important to keep in mind that
\[
\delta_j^\pm = \mathcal O (\dot s).
\]

We also have the expansions
\[
\begin{array}{rllll}
\partial_s \psi_j (s_j^*) & =&  \partial_s \psi_j (s_j) - K_g(s_j) \, (1 + \delta^+_j) \, \dot s_j  - \partial_s K_g(s_j) (1+\delta^+_j)  \\[2mm]
& + & \displaystyle K_g(s_j)\,  \partial_s \phi_j(s_j) \, \frac{\dot s_j^2}{2} + \mathcal O(\dot s^3),
\end{array}
\]
and
\[
\begin{array}{rllll}
\partial_s \psi_j (s_{j+1}^*) & =&  \partial_s \psi_j (s_{j+1}) - K_g(s_{j+1}) \, (1 + \delta^-_{j+1}) \, \dot s_{j+1} - \partial_s K_g(s_{j+1})\,  (1+\delta^-_{j+1})  \\[2mm]
& + & \displaystyle K_g(s_j)\,  \partial_s\phi_j (s_{j+1})\, \frac{\dot s_{j+1}^2}{2} + \mathcal O(\dot s^3).
\end{array}
\]

Since the functions $\psi_{j-1}$ and $\psi_{j}$ take the same value (equal to $1$) at $s_j^*$, the restrictions of $\psi_j$ to $[s_j^*, s_{j+1}^*$] can be patched together to obtain a function which is continuous on $\gamma$, smooth away from the $s_j^*$ and which solves $\mathfrak J_\gamma \psi_j = 0$ on $(s^*_j, s_{j+1}^*)$ with boundary data $\psi_j (s_j^*)=\psi_j(s_{j+1}^*)=1$.

\medskip

We are now in a position to compute the expansion of the functional $\mathcal H_n$ in powers of $\dot s_j$. We start with an integration by parts which leads to
\[
\begin{array}{rlll}
\mathcal H_n (s_1^*, \ldots, s_n^*) & = & \displaystyle \sum_{j=1}^n \int_{s_j^*}^{s_{j+1}^*}\left( (\partial_s \psi_j)^2 - K_g \, \psi_j^2\right) \, ds\\[3mm]
& = & \displaystyle \sum_{j=1}^n \left( \partial_s \psi_j(s_j^*)- \partial_s \psi_j(s_{j+1}^*)\right).
\end{array}
\]
Using the above expansions, we get
\[
\begin{array}{rllll}
\mathcal H_n (s_1^*,\ldots, s_n^*) & = & \mathcal H_n (s_1,\ldots, s_n) - \displaystyle   \sum_{j=1}^n \left(\partial_s \eta_j(s_{j+1}) - \partial_s\eta_{j}(s_j)\right) \\[3mm]
& + & \displaystyle \sum_{j=1}^n K_g(s_{j+1}) \, \delta_{j+1}^- \, \dot s_{j+1} - \sum_{j=1}^n K_g(s_j) \, \delta_j^+ \, \dot s_j  \\[3mm]
& + & \displaystyle \sum_{j=1}^n K_g(s_{j+1})\, \partial_s\phi_j(s_{j+1}) \frac{\dot s_{j+1}^2}{2}\\[3mm]
& - & \displaystyle  \sum_{j=1}^n K_g(s_j)\, \partial_s\phi_j(s_j) \frac{\dot s_j^2}{2} + \mathcal O(\dot s^3)\\[3mm]
& = & \mathcal H_n (s_1,\ldots, s_n) - \displaystyle   \sum_{j=1}^n \left(\partial_s \eta_j (s_{j+1}) - \partial_s\eta_{j} (s_j)\right) \\[3mm]
& - & \displaystyle \sum_{j=1}^n K_g(s_j) \, (\delta_{j}^+ - \delta_j^-)\, \dot s_j  \\[3mm]
& - & \displaystyle \sum_{j=1}^n K_g(s_j) \, \partial_s (\phi_j(s_j) -  \phi_{j-1} )(s_j)\, \frac{\dot s_j^2}{2} + \mathcal O(\dot s^3).
\end{array}
\]

Observe that, since $\phi_j$ and $\eta_j$ are both solutions of the same second order linear ordinary differential equation, we have
\[
0 = \int_{s_j}^{s_{j+1}} \left( \eta_j \mathfrak J_\gamma     \phi_j - \phi_j  \mathfrak J_\gamma     \eta_j \right) \, ds =
(\eta_j  \partial_s \phi_j - \phi_j  \partial_s \eta_j)(s_j) - (\eta_j  \partial_s \phi_j - \phi_j  \partial_s\eta_j)(s_{j+1}).
\]
Hence
\[
\partial_s \eta_j (s_{j+1}) - \partial_s \eta_{j}(s_j) =\delta_{j+1}^- \, \partial_s \phi_j (s_{j+1}) -  \delta^+_j \, \partial_s\phi_{j}(s_j) .
\]
Using this, we get with little work
\[
\begin{array}{rllll}
\mathcal H (s^*_1,\ldots,  s_n^*)& = & \mathcal H (s_1,\ldots,  s_n)  \displaystyle  + \frac{1}{2} \sum_{j=1}^n \partial_s \left(\phi_j + \phi_{j-1}\right)(s_j)\,  (\delta^+_j-\delta_j^-) \\[3mm]
 & +  &  \displaystyle \frac{1}{2} \sum_{j=1}^n \partial_s \left(\phi_j - \phi_{j-1}\right)(s_j)\,  (\delta^+_j+\delta_j^-) - \displaystyle \sum_{j=1}^n K_g(s_j) \, (\delta_{j}^+ -\delta_j^-) \, \dot s_j \\[3mm]
 & - & \displaystyle \sum_{j=1}^n K_g(s_j) \, \partial_s (\phi_j- \phi_{j-1} )(s_j)\, \frac{\dot s_j^2}{2} + \mathcal O(\dot s^3).
\end{array}
\]
Finally, we use (\ref{eq:2.4}) and (\ref{eq:2.5}) to conclude that
\begin{equation}
\begin{array}{llll}
\mathcal H_n (s^*_1, \ldots, s^*_n) &= & \displaystyle\mathcal H_n (s_1, \ldots, s_n) - \sum_{j=1}^n ((\partial_s\phi_j)^2- (\partial_s \phi_{j-1})^2)(s_j)\, \dot s_j  \\[3mm]
 & - & \displaystyle \sum_{j=1}^n  \partial_s (\phi_{j} - \phi_{j-1})  (s_j) \, \partial_s (\eta_{j} + \eta_{j-1})(s_j)\, \frac{\dot s_j^2}{2} \\[3mm]
 & + & \displaystyle
 \sum_{j=1}^n K_g(s_j) \, \partial_s ( \phi_j- \phi_{j-1} )(s_j)\, \frac{\dot s_j^2}{2} \\[3mm]
 & - &  \displaystyle \sum_{j=1}^n ((\partial_s\phi_j)^2- (\partial_s \phi_{j-1})^2)(s_j)\, \mathcal O (\dot s^2) + \mathcal O(\dot s^3).
\end{array}
\label{eq:2.6}
\end{equation}

Since we have assumed that $(s_1, \ldots, s_n)$ is a critical point of $\mathcal H_n$, the linear term in $\dot s_j$ has to vanish, and hence we get
\[
(\partial_s \phi_j)^2(s_j) = (\partial_s\phi_{j-1})^2(s_j).
\]
But, $\phi_{j-1}$ and $\phi_{j}$ are minimal at $s_j$, therefore, $\partial_s\phi_j (s_j) \geq 0$ and $\partial_s\phi_{j-1} (s_j)\leq 0$ and this implies that
\[
\partial_s (\phi_j + \phi_{j-1}) (s_j)=0,
\]
for all $j=1, \ldots, n$. Since $K_g >0$ along $\gamma$, it is not possible for $\partial_s \phi_j (s_j)$ to be equal to $0$ since if this were the case then the functions $\psi_j$ on $(s_{j-1}, s_j]$ and $\psi_{j+1}$ on $[s_j, s_{j+1})$ can be glued together to provide a solution of $\mathfrak J_\gamma \psi =0$ on $(s_{j-1}, s_{j+1}$ which has a positive local minimum at $s_j$, which is a contradiction.
This proves that $\Phi$ is a Bouncing Jacobi Field.
\end{proof}

To proceed, let us give a necessary condition for the existence of a Bouncing Jacobi Field.
\begin{lemma}
\label{le:2.6}
Assume that
\[
|\gamma|^2 \, \inf_\gamma K_g > \pi^2\, n^2.
\]
Then $\gamma$ does not carry any Bouncing Jacobi Field with $n$ minimums.
\end{lemma}
\begin{proof}
We argue by contradiction and assume that there was a Bouncing Jacobi Field $\Phi$ with $n$ consecutive local minimums at  $s_1 < \ldots < s_n$. We can find an index $j$ such that
\[
n \, \sqrt{\inf K_g} \, |s_{j+1}-s_j|\geq|\gamma| \sqrt{\inf K_g} > n \, \pi.
\]
It follows that the equation
\[
\left( \partial_s^2 + \inf K_g \right) \,\eta =0,
\]
has a nontrivial solution $\eta$ which vanishes at two distinct points $\overline s < \underline s$ in the interval $(s_j, s_{j+1})$. Now,
\[
\mathfrak J_\gamma \Phi =0,
\]
in $(s_j, s_{j+1})$, with $\Phi (s_j) = \Phi (s_{j+1}) = 1$.
Since $\Phi >1$ in $(\underline s, \overline s)$, one can choose a constant $c >0$ such that  $\Phi - c \, \eta$ has a local minimum in $(\underline s, \overline s)$.
This implies that $\Phi - c \, \eta \equiv 0$ and this is a contradiction since this function is not zero at $\overline s$ and $\underline s$.
\end{proof}

The proof of Theorem \ref{th:2.2} is now complete.
This also completes the proof of Theorem \ref{th:1.4} mentioned in the first section.
\end{proof}

\begin{remark}
\label{re:2.7}
In the special case where $n=1$ and where (\ref{eq:2.3}) holds, there exist at least two distinct Bouncing Jacobi Fields with one local minimum since, in this case, $s \mapsto \mathcal H_1(s)$ is a smooth periodic function which has at least one minimum and one maximum.
\end{remark}
\section{The index of Bouncing Jacobi Fields}
Observe that, in the case where the Gauss curvature is constant along $\gamma$ and if $\gamma$ is assumed to be non-degenerate, then inequality $\left(\ref{th:1.5}\right)$ tells us that $n$, the number of minimum of a Bouncing Jacobi Field, satisfies
\[
n \geq \text{Ind}\left(\gamma\right),
\]
where we recall that $\text{Ind}\left(\gamma\right)$ denotes the Morse index of the geodesic $\gamma$.

\medskip

In the case where the Gauss curvature is not constant, the relation between $n$ and the index of the geodesic is not clear.
In this section, we will analyze the properties of the index of the Bouncing Jacobi Field.
These properties will be useful in  our later analysis of the Jacobi-Toda system.
As a by product,  in  Corollary~\ref{co:5.6}, we will also prove a weaker lower bound for $n$ in terms of the index of $\gamma$, which has independent interest.

\medskip

To proceed, we will now compute the the second variation of $\mathcal H_n$ at a critical point.

\begin{proposition}
\label{pr:3.1}
Assume that $\Phi$ is Bouncing Jacobi Field with $n$ minimums at $0\leq s_1< \ldots < s_n < |\gamma|$. As usual, we agree that $s_{n+1}:= s_1 +|\gamma|$ and $s_{0} = s_{n} - |\gamma|$ and we define
\[
N_j := \partial_s \Phi (s_j^+)- \partial_s \Phi (s_j^-).
\]
Then, the following expansion holds
\[
\begin{array}{lllll}
\displaystyle {\mathcal H}_n \left( s_1+ \frac{\sigma_1}{N_1}, \ldots , s_n +\frac{\sigma_n}{N_n}\right) = {\mathcal H}_n (s_1, \ldots , s_n )   \\[3mm]
\hspace{25mm} \displaystyle  - \frac{1}{4} \sum_{j=1}^n {\bigg\{} \left( \partial_s \eta (s_{j}^-) + \partial_s \eta (s_j^+) \right) \, \sigma_j + \frac{4 \, K_g(s_j) }{N_j} \, \sigma_j^2 {\bigg\}} + \mathcal O(\sigma_j^3),
\end{array}
\]
or equivalently
\[
\begin{array}{lllll}
\displaystyle \mathcal H_n \left( s_1+ \frac{\sigma_1}{N_1}, \ldots , s_n +\frac{\sigma_n}{N_n}\right) = {\mathcal H}_n (s_1, \ldots , s_n )    \\[3mm]
\hspace{20mm} \displaystyle  + \frac{1}{4} \sum_{j=1}^n {\bigg\{} \int_{s_j}^{s_{j+1}} \left( (\partial_s  \eta)^2 - K_g \, \eta^2\right) \, ds - \frac{4 \, K_g(s_j) }{N_j} \, \sigma_j^2 {\bigg\}} +\displaystyle \sum_{j=1}^n \mathcal O(\sigma_j^3),
\end{array}
\]
where the function $\eta : \gamma -\{s_1, \ldots, s_n\} \to \R$ is the (unique) solution of
\[
\left\{
\begin{array}{rlllll}
\mathfrak J_\gamma  \eta & =& 0 \qquad \text{on} \qquad \gamma -\{s_1, \ldots, s_n\},
\\[3mm]
\eta (s_j^+) = - \eta (s_j^-) & =&  \sigma_{j} \qquad \text{for all} \quad j=1, \ldots, n.
\end{array}
\right.
\]
\end{proposition}
\begin{proof}
We keep the notations of the proof of Proposition~\ref{pr:2.5} and start with (\ref{eq:2.6}) which we specialize to the case where $\Phi$ is Bouncing Jacobi Field. In this case, the formula simplifies and we get
\[
\begin{array}{llll}
\mathcal H_n (s^*_1, \ldots, s^*_n) &= & \displaystyle\mathcal H_n (s_1, \ldots, s_n)  \\[3mm]
& - & \displaystyle \sum_{j=1}^n  \partial_s (\phi_{j} - \phi_{j-1})  (s_j) \, \partial_s (\eta_{j} + \eta_{j-1})(s_j)\, \frac{\dot s_j^2}{2} \\[3mm]
 & + & \displaystyle
 \sum_{j=1}^n K_g(s_j) \, \partial_s( \phi_j- \phi_{j-1} )(s_j)\, \frac{\dot s_j^2}{2} + \mathcal O(\dot s^3).
\end{array}
\]
where we recall that $\dot s_j := s^*_j-s_j$. Using the fact that
\[
\delta_j^\pm = \pm \frac{N_j}{2} \dot s_j + \mathcal O (\dot s^2),
\]
where we recall that
\[
N_j := \partial_s (\phi_j -\phi_{j-1})(s_j),
\]
this can be rewritten as
\[
\begin{array}{lllll}
\mathcal H (s^*_1,\ldots,  s_n^*)  & =&  \mathcal H (s_1,\ldots,  s_n) \\[3mm]
& - & \displaystyle \frac{1}{4} \, \sum_{j=1}^n  N_j^2 \, \bigg\{\partial_s (\dot \eta_j + \dot \eta_{j-1})(s_j) \, \dot s_j + \frac{4 \, K_g(s_j) }{N_j} \dot s_j^2 \bigg\} + \sum_{j=1}^n \mathcal O(\dot s_j^3),
\end{array}
\]
where $\dot \eta_j$ is the solution of $\mathfrak J_\gamma    \dot \eta_j =0$ on $(s_j, s_{j+1})$, with boundary data $\dot s_j$ at $s_j$ and $\dot s_{j+1}$ at $s_{j+1}$.

\medskip

Since we also have
\[
\int_{s_j}^{s_{j+1}} \left( (\partial_s \dot \eta_j)^2-K_g\, \dot \eta_j^2\right)\, ds = - \left( \partial_s\eta_j(s_{j+1}) \, \dot s_{j+1} + \partial_s\eta_j (s_j) \, \dot s_j \right),
\]
we get the alternative formula
\[
\begin{array}{llll}
\mathcal H (s_1^*, \ldots, s_n^*) =  \mathcal H (s_1,\ldots,  s_n) \\[3mm]
\hspace{21mm}\, + \displaystyle \frac{1}{4} \sum_{j=1}^n N_j^2 \, \bigg\{\int_{s_j}^{s_{j+1}} \left((\partial_s \dot \eta_j)^2 - K_g (\dot \eta_j)^2\right) \, ds - \frac{4 \, K_g(s_j) }{N_j} \dot s_j^2 \bigg\} + \mathcal O(\dot s_j^3).
\end{array}
\]

This completes the proof.
\end{proof}

We now introduce the definition of a non-degenerate Bouncing Jacobi Field.

\begin{definition}
\label{de:3.2}
A Bouncing Jacobi Field $\Phi$ with $n$ minimums at the points $s_1 < \ldots< s_n$ (being understood that $s_{n+1}:=s_1+|\gamma|$ and $s_0:=s_n-|\gamma|$) is said to be non-degenerate, if the following system does not have any non-trivial solution:
\[
\left\{
\begin{array}{rllll}
\mathfrak J_\gamma    \eta & = & 0 \quad \text{in} \quad  \gamma -\{s_1, \ldots, s_n\},\\[3mm]
\eta (s^+_j) +\eta (s^-_j) & = & 0 \quad \text{for all} \quad   j=1, \ldots, n ,\\[3mm]
\displaystyle \partial_s \eta (s_j^+) + \partial_s \eta (s_{j}^-) + \displaystyle 2 K_g \, \frac{\eta (s_j^+) - \eta (s_j^-)}{N_j}& =& 0 \quad \text{for all} \quad j= 1, \ldots, n,
\end{array}
\right.
\]
where $N_j:= \partial_s \Phi (s_j^+)-\partial_s \Phi (s_j^-)$.
\end{definition}
Observe that the operator which appears in this definition also appears in the second variation of the functional $\mathcal H_n$ which is given in Proposition~\ref{pr:3.1}.

\medskip

The non-degeneracy condition will now be used to prove that we can embed the Bouncing Jacobi Field $\Phi$ as a $2n$ dimensional family of continuous, piecewise $\mathcal C^1$ Jacobi Fields  whose $n$ local minimum have a prescribed value close to $1$ and have the property that the sum of the right and left derivatives at these points can also be prescribed to some value close to $0$.
More precisely, we have the:
\begin{proposition}
\label{pr:3.3}
Assume that $\Phi$ is a non-degenerate Bouncing Jacobi Field with $n$ minimums at the points $s_1, \ldots, s_n \in \gamma$.
Then, for all $\delta_1, \ldots, \delta_n \in \R$ and $\theta_1, \ldots , \theta_n \in \R$ close enough to $0$, there exists $\tilde s_1, \ldots, \tilde s_{n}$ in $\gamma$ and $\tilde \Phi$ solution of:
\begin{equation}
\label{eq:tphi}
\left\{
\begin{array}{rllll}
   \mathfrak J_\gamma \tilde \Phi & = & 0 \quad \text{on} \quad \gamma - \{\tilde s_1, \ldots, \tilde s_{n}\} , \\[3mm]
    \tilde \Phi (\tilde s_j) & = & 1+\delta_j \quad \text{for} \quad  j=1, \ldots, n ,\\[3mm]
\partial_s \tilde \Phi (\tilde s_j^+) + \partial_s \tilde \Phi (\tilde s_j^-) & = & \theta_j \quad \text{for} \quad j=1, \ldots, n .
\end{array}
\right.
\end{equation}
Moreover, $\tilde \Phi$ depends smoothly on the $\delta_j$ and the $\theta_j$, and $\tilde \Phi =\Phi$ when $\delta_j=\theta_j = 0$ for $j=1,\ldots, n$. Finally, the points $\tilde s_1, \ldots, \tilde s_n$ and the function $\tilde \Phi$ depend smoothly on the parameters $\delta_1, \ldots, \delta_n \in \R$ and $\theta_1, \ldots , \theta_n \in \R$ (provided these are close enough to $0$).
\end{proposition}
\begin{proof}
Keeping the notation of the proof of Proposition~\ref{pr:2.5}, we can write
\[
\psi_j (s_j + s) =  1+ \delta_j^+ + \partial_s (\phi_j+\eta_j) (s_j)\, s - K_g(s_j) (1+\delta_j^+) \frac{s^2}{2} + \mathcal O(s^3),
\]
and
\[
\psi_{j-1} (s_j + s) =  1+ \delta_j^- + \partial_s (\phi_{j-1} +\eta_{j-1}) (s_j)\, s - K_g(s_j) (1+\delta_j^-) \frac{s^2}{2} + \mathcal O(s^3),
\]
Given the $\theta_j$ and the $\delta_j$, we determine $\tilde s_j$ close to $s_j$, and $\delta_j^\pm$ close to $0$ such that
\[
\psi_{j-1} (\tilde s_j) = \psi_{j-1} (\tilde s_j) =1+\delta_j,
\]
and
\[
\partial_s (\psi_j + \psi_{j-1})(\tilde s_j) = \theta_j.
\]
With little work, we get
\[
\psi_{j} (\tilde s_j) = \psi_{j-1} (\tilde s_j)  =  1 + \frac{\delta_j^+ + \delta_j^-}{2} + \sum_i \mathcal O (\delta_i^2 +\theta_i^2),
\]
\[
\partial_s (\phi_{j}- \phi_{j-1}) (s_j) \, (\tilde s_j -s_j) + \delta_j^+ - \delta_j^- = \sum_i \mathcal O (\delta_i^2 +\theta_i^2),
\]
and
\[
\partial_s \left( \psi_j + \psi_{j-1} \right) (\tilde s_j) = \displaystyle \partial_s (\eta_j +  \eta_{j-1}) (s_j) - 2 K_g(s_j) \, (\tilde s_j-s_j) \displaystyle +   \sum_i \mathcal O(\delta_i^2 + \theta_i^2).
\]
So our equations become
\[
\delta_j^+ + \delta_j^- = 2\, \delta_j+ \sum_i \mathcal O (\delta_i^2 +\theta_i^2),
\]
and
\[
\displaystyle \partial_s (\eta_j +  \eta_{j-1}) (s_j) + 2 \, \frac{K_g(s_j)}{\partial_s (\phi_j- \phi_{j-1})(s_j)} (\delta_j^+ - \delta_j^-) = \theta_j + \displaystyle   \sum_i \mathcal O(\delta_i^2 + \theta_i^2).
\]

The non-degeneracy of the Bouncing Jacobi Field ensures that the linear map
\[
(\delta^-_j, \delta^+_j)_{j=1, \ldots, n} \mapsto \left(\delta^+_j + \delta^-_j, \partial_s (\eta_j + \eta_{j-1} ) (s_j) +  K_g(s_j) \, \displaystyle \frac{\delta^+_j-\delta^-_j}{N_j} \right)_{j=1, \ldots, n},
\]
defined from $(\R^2)^n$ into $(\R^2)^n$, is one-to-one and hence, it is invertible. Recall that the functions $\eta_j$ solve $\mathfrak J_\gamma \eta_j =0$ in $(s_j, s_{j+1})$ with $\eta_j(s_j) = \delta_j^+$ and $\eta_j(s_{j+1})=\delta^-_{j+1}$ and hence they depend linearly on the $\delta^\pm_j$.

\medskip

The inverse function Theorem can be applied to determine the $\delta^\pm_j$ as functions of the $\delta_i$ and $\theta_i$ provided these are close to $0$. This completes the proof of the existence of $\tilde \Phi$ solution of (\ref{eq:tphi}).
\end{proof}

There is a functional which is closely related to the second variation of $\mathcal H_n$ and which will be very useful when we will need to determine the Morse index of the solutions of the Jacobi-Toda system and then of the solutions of the Allen-Cahn equation.

\medskip

Given a Bouncing Jacobi Field $\Phi$ with $n$ minimums at $s_1 < \ldots < s_n$, we define
\begin{equation}
    \label{eq:3.2}
\begin{array}{lllll}
\mathcal Q (\eta ) := \displaystyle \sum_{j=1}^n {\bigg\{} \int_{s_j}^{s_{j+1}} \left( (\partial_s \eta )^2 - K_g \, \eta^2\right) \, ds - \frac{K_g(s_j) }{N_j} \, \left(\eta (s_j^+) - \eta (s_j^-)\right)^2 {\bigg\}},
\end{array}
\end{equation}
for all functions $\eta : \gamma-\{s_1, \ldots, s_n\} \to \R$ such that
$\eta_{|(s_j, s_{j+1})} \in H^1(s_j, s_{j+1})$ and  $\eta (s_{j}^+) + \eta_j(s_j^-) =0$ for all $j=1, \ldots, n$. As usual, $N_j := \partial_s \Phi (s_j^+)- \partial_s \Phi (s_j^-)$.
We now prove that the index of $\mathcal Q$ at $\Phi$ (which is the maximal dimension of the space of functions over which $\mathcal Q$ is negative definite), is equal to the index of the Bouncing Jacobi Field $\Phi$ as defined in the introduction.
\begin{proposition}
\label{pr:3.4}
Given a Bouncing Jacobi Field $\Phi$ with $n$ minimums at $s_1 < \ldots < s_n$, its index is equal to the index of $\mathcal Q$ at $\Phi$
\end{proposition}
\begin{proof}
Let $E$ be the vector space spanned by the functions $\xi^i : \gamma -\{s_1, \ldots, s_n\} \to \R$, for $i=1, \ldots, n$, where $\xi^i$ is a solutions of the Jacobi equation $\mathfrak J_\gamma    \xi^i =0$ on each $(s_j, s_{j+1})$ with boundary data
\[
\xi^i (s_j^-)=\xi_j(s_j^+) =0 \quad \text{ for all }  \quad j\neq i \quad \text{ and } \quad  \xi^i (s_i^+) = - \xi^i(s_i^-)=1.
\]
The existence of $\xi^i$ follows from the fact that, since we have a Bouncing Jacobi Field, the operator $\mathfrak J_\gamma$, restricted to $H^1_0(s_j, s_{j+1})$, is non-degenerate.
It follows from the alternative formula in Proposition~\ref{pr:3.1}, that the index of the quadratic form $\mathcal Q$, restricted to $E$, is equal to the index of $\mathcal H_n$ at the point $(s_1, \ldots, s_n)$.

\medskip

Next, we define $F$ to be the vector space of functions $\psi:  \gamma -\{s_1, \ldots, s_n\} \to \R$ such that $\psi_{|(s_j, s_{j+1})}\in H^1_0(s_j, s_{j+1})$, for all $j=1, \ldots, n$.
The existence of the Bouncing Jacobi Field $\Phi$ implies that the operator $\mathfrak J_\gamma$ has index $0$ and nullity $0$ on each $(s_j, s_{j+1})$), therefore, the restriction of the quadratic form $\mathcal Q$ to $F$ is positive definite.

\medskip

Now, observe that any $\eta : \gamma -\{s_1, \ldots, s_n\} \to \R$ with $\eta_{|(s_j, s_{j+1})} \in H^1(s_j, s_{j+1})$ and $\eta (s_{j}^+) + \eta (s_j^-) =0$, for $j=1, \ldots, n$, can be uniquely decomposed as
\[
\eta  = \xi+ \psi,
\]
where  $\xi \in E$ and $\psi  \in F$. Moreover, a simple integration by parts shows that
\[
\mathcal Q (\eta) = \mathcal Q (\xi) + \mathcal Q(\psi).
\]
The index of $\mathcal Q$ restricted to $E$ being equal to the index of   $\mathcal{H}_{n}$ and $\mathcal Q$ being positive definite on $F$, we conclude that the index of $\mathcal{Q}$ is equal to the index of $\mathcal{H}_{n}$.
\end{proof}
\section{The Jacobi-Toda equation}
With some existence results for Bouncing Jacobi Fields at hand, we prove in this section the existence of solutions $\Psi : \gamma \to \R$ to the so-called Jacobi-Toda equation:
\begin{equation}
    \label{eq:4.1}
\varepsilon^2 \, \mathfrak J_\gamma \Psi + \bar{c}^2\, e^{-2\Psi} = 0,
\end{equation}
where $\Psi : \gamma \to \R$ and $\bar c >0$ is a constant which will be defined later on.

\medskip

Actually we are interested in solutions of this equation whose Morse index is bounded uniformly in $\varepsilon$ since, as we will see later on, these solutions will hopefully lead to the construction of solutions of the Allen-Cahn equation which have $2$ transition layers converging to $\gamma$ and also have uniformly bounded Morse indices.
We will show that solutions of the Jacobi-Toda equation can be obtained by desingularizing a Bouncing Jacobi Fields using solutions of a Toda equation.

\medskip

For all $\varepsilon >0$, it will be convenient to define $\alpha_\varepsilon > 0$ to be the unique solution of
\begin{equation}
\label{eq:4.2}
\alpha_\varepsilon \, e^{\alpha_\varepsilon} = \frac{1}{\varepsilon}.
\end{equation}
It is an easy exercise to check that, as a function of $\varepsilon$, $\alpha_\varepsilon$ can be expanded as
\[
\alpha_\varepsilon = - \ln \varepsilon - \ln \, (- \ln \varepsilon) + \mathcal O\left(\frac{\ln \, (-\ln \varepsilon)}{\ln \varepsilon}\right),
\]
as $\varepsilon$ tends to $0$.

\medskip

In the following, we agree that $C$ are positive constants, which  may depend from one inequality to the next, but which can be fixed large enough independently of $\varepsilon$.
These constants only depend on the geometry of our problem, namely the geodesic, curvature of the surface along the geodesic and the Bouncing Jacobi Field we consider to start the construction.
Similarly, we agree that $c$ are positive constants which may also  depend from one inequality to the next, but can be fixed small enough independently of $\varepsilon$.
Again, the constants $c$ only depend on the geometry of our problem and and the Bouncing Jacobi Field we consider.

\medskip

Given a non-degenerate Bouncing Jacobi Field $\Phi$ with $n$ minimums at $s_1 < \ldots < s_n$, we first explain how to perturb $\alpha_\varepsilon \, \Phi$ into a $2n$-dimensional family of solutions of (\ref{eq:4.1}) which are defined away from small intervals centered at the $s_j$'s.

\medskip

Assume we are given $\delta_j \in \R$ and $\theta_j \in \R$ such that
\[
|\delta_j| + \alpha_\varepsilon \, |\theta_j| \leq \frac{C}{\alpha_\varepsilon},
\]
where $C >0$ is a constant fixed large enough independently of $\varepsilon$.

\medskip

Proposition~\ref{pr:3.3} provides a continuous function $\tilde \Phi = \tilde \Phi (\delta_1, \ldots, \delta_n, \theta_1, \ldots, \theta_n)$ which has $n$ minimums at points $\tilde s_1 < \ldots < \tilde s_n$ (each of which is a function of the $\delta_j$'s and the $\theta_j$'s), and which solves the equation $\mathfrak J_\gamma \, \tilde \Phi =0$ in each $(\tilde s_j, \tilde s_{j+1})$ with initial data given by
\[
\tilde \Phi (\tilde s_j) = 1 + \delta_j \qquad \text{and} \qquad  \partial_s \tilde \Phi (\tilde s_j^+)+ \partial_s \tilde \Phi (\tilde s_j^-) = \theta_j.
\]

Observe that the function $\tilde \Phi$ and the $\tilde s_j$'s depend (smoothly) on the parameters $\delta_j$ and $\theta_j$, for $j=1, \ldots, n$ and we have the estimate
\[
\sup_j |\tilde s_j - s_j| \leq \frac{C}{\alpha_ \varepsilon}.
\]
We are now going to perturb the function $\alpha_\varepsilon \, \tilde \Phi$ into a function $\Psi$ which solves the Jacobi-Toda equation (\ref{eq:4.1}), away from some small intervals of size $2C/\alpha_\varepsilon$, centered at the $\tilde s_j$'s. We are solving a second order ordinary differential equation and since we do not prescribe the boundary data of $\Psi$ at the ends of the intervals, we need to address the issue of the non-uniqueness of the solution.
This is the reason why, for each $j=1, \ldots, n$, we ask that
\[
\Psi (\hat s_j) =  \alpha_\varepsilon \, \tilde \Phi (\hat s_j) \qquad \text{and} \qquad \partial_s \Psi (\hat s_j)  = \alpha_\varepsilon \partial_s\tilde\Phi (\hat s_j),
\]
where $\hat s_j$ is the middle of the interval $(\tilde s_j, \tilde s_{j+1})$.
This choice might seem quite arbitrary and the sceptic reader might ask if this is a reasonable choice. Let us just recall that we are solving a second order ordinary differential equation, so the space of solutions is $2$-dimensional. But in our construction we have two parameters $\delta_j$ and $\theta_j$ which are free and can be used to parameterize, close to $\alpha_\varepsilon \, \Phi$, the set of all solutions to the equation.

\medskip

The existence of a solution $\Psi$ close to $\alpha_\varepsilon \, \tilde \Phi$ follows from a perturbation argument, the rational being that, away from small intervals centered at the $\tilde s_j$'s, the function $\alpha_\varepsilon \, \tilde \Phi$ is much larger than $\alpha_\varepsilon$ and hence the term $\bar c^2\, e^{- 2 \alpha_\varepsilon \, \tilde \Phi}$ in the equation is small enough to ensure the perturbation argument to work.

\medskip

To be more precise, we have the:
\begin{proposition}
\label{pr:4.1}
There exist constants $C >0$ and, for all $\varepsilon >0$ small enough, there exists a function $\psi = \psi (\varepsilon, \delta_1, \ldots , \delta_n, \theta_1, \ldots \theta_n)$ such that
\[
\Psi := \alpha_\varepsilon\, (\tilde \Phi +\psi),
\]
solves $\mathfrak J_\gamma \, \Psi =0$ in $(\tilde s_j + \frac{C}{\alpha_\varepsilon} , \tilde s_{j+1} - \frac{C}{\alpha_\varepsilon})$ with
\[
\Psi =  \alpha_\varepsilon\tilde \Phi  \qquad \text{and} \qquad  \partial_s \Psi = \alpha_\varepsilon \partial_s\tilde \Phi ,
\]
at $\hat s_j$, for all $j=1, \ldots, n$. Moreover, the function $\psi$ depends smoothly on the $\delta_j$'s and $\theta_j$'s and satisfies the estimates
\[
|\psi| + |\partial_{\delta_j} \psi| + |\partial_{\theta_j} \psi| \leq C \, e^{- \frac{\alpha_\varepsilon}{C} \, |s-\tilde s_j|} ,
\]
and
\[
|\partial_s \psi|  + |\partial_{\delta_j} \partial_s \psi| + |\partial_{\theta_j} \partial_s  \psi| \leq C \, \alpha_\varepsilon \, e^{- \frac{\alpha_\varepsilon}{C} \, |s-\tilde s_j|} .
\]
\end{proposition}
\begin{proof}
If $f : \gamma-\{\tilde s_1, \ldots, \tilde s_n\} \to \R$, we define the operator $\mathcal S$  by
\[
\mathcal S (f) (s) := - \tilde \Phi (s) \int_{\hat s_j}^s  \frac{1}{(\tilde \Phi (t))^2} \left( \int_{\hat s_j}^t \tilde \Phi (t')\, f(t') \, dt' \, \right) \, dt ,
\]
for all $s \in (\tilde s_j, \tilde s_{j+1})$.
Since $\tilde \Phi$ is a Jacobi field which does not vanish on $(\tilde s_j, \tilde s_{j+1})$, the operator ${\mathcal S}$ is well defined and we have
\[
- (\partial_s^2 + K_g)\, {\mathcal S}(f) = f,
\]
on $\gamma-\{\tilde s_1, \ldots, \tilde s_n\}$ with ${\mathcal S}(f) (\hat s_j) =\partial_s {\mathcal S}(f) (\hat s_j) =0$.

\medskip

We can then rephrase the problem as finding $\psi$, solution of
\[
\psi = -\frac{1}{\alpha_{\varepsilon}} \mathcal S \left( \frac{\bar c^2}{\varepsilon^2} \, e^{-2 \alpha_\varepsilon \, (\tilde \Phi + \psi)}\right).
\]
It is a simple (but tedious) exercise to check that, there exists $C >0$ large enough, independent of $\varepsilon$, such that this nonlinear equation has a solution well defined on each $(\tilde s_j + \frac{C}{\alpha_\varepsilon}, \tilde s_{j+1} - \frac{C}{\alpha_\varepsilon})$.
The solution can be obtained applying a fixed point argument for contraction mappings. Moreover, using (\ref{pr:4.2}), we get the estimates
\[
\alpha_\varepsilon\, |\psi (s)| + | \partial_s \psi (s) |\leq C \,  \alpha_\varepsilon \, e^{-\frac{\alpha_\varepsilon}{C} \, |s-\tilde s_j|}.
\]
Similar estimates can be derived for the partial derivative of $\psi$ with respect to the parameters $\delta_j$ and $\theta_j$.
\end{proof}

We now turn to the second ingredient in our construction.
To start with, observe that the Toda equation
\begin{equation}
\varepsilon^2 \, \partial_s^2 T + \bar c^2\, e^{-2 T} =0,
\label{eq:4.3}
\end{equation}
where $T : \R \to \R$, can be explicitly integrated and all solutions are given by
\[
s \mapsto \ln \left( \cosh \left( \frac{\kappa}{\varepsilon} \, (s-\bar s) s \right)\right) -\ln \kappa + \ln \bar c,
\]
for some $\kappa >0$ and some $\bar s \in \R$.
Our next result states that one can perturb this function into a solution of
\begin{equation}
\label{eq:vkappa}
\varepsilon^2 \, \mathfrak J_\gamma   T + \bar c^2 \, e^{-2T} = 0
\end{equation}
on some small interval centered at $\bar s$, with initial conditions given by
\[
T(\bar s) = \partial_s T(\bar s) = 0.
\]
The rational being that in the regime where the term $\varepsilon^2 \, K_g T$ is small compared to the other terms, a perturbation argument will lead to a solution $v$ of (\ref{eq:vkappa}) close to $T_0$. Again, the initial conditions we impose might seem arbitrary but the interested reader will check that playing with the parameters $\kappa$ and translation $\bar s$, one can prescribe any small value for the function $T$ and its first derivative at a fixed point $\bar s$.

\medskip

We define
\[
T_0 (s) := \ln \left( \cosh \left( \frac{\kappa}{\varepsilon} \, (s-\bar s) s \right)\right) -\ln \kappa + \ln \bar c + K_g(\bar s) \, \ln \kappa \, \frac{(s-\bar s)^2}{2}.
\]
With this definition at hand, we have the:
\begin{proposition}
\label{pr:4.2}
For all $\varepsilon >0$ small enough, all $\bar s \in \R$ and all $\kappa >0$ satisfying
\[
\frac{1}{C}  \, \alpha_\varepsilon \leq \frac{\kappa}{\varepsilon} \leq  C \, \alpha_\varepsilon,
\]
there exists $c >0$ and a unique  $T = T (\varepsilon, \kappa, \bar s)$ solution of (\ref{eq:vkappa}) on $[\bar s - c, \bar s + c]$ which depends smoothly on $\varepsilon$, $\bar s$ and $\kappa$. Moreover
\[
|(T-T_0) (s)| + |\kappa \, \partial_\kappa (T-T_0) (s)|  + |\partial_{\bar s} (T-T_0) (s)| \leq  C \, \left(\alpha_\varepsilon \, |s-\bar s|^3+ \frac{|s-\bar s|}{\alpha_\varepsilon}\right) ,
\]
and
\[
|\partial_s (T-T_0)(s)| + |\kappa \, \partial_\kappa \partial_s (T-T_0) (s)|  + |\partial_{\bar s}\partial_s (T-T_0) (s)|\leq  C \, \left( \alpha_\varepsilon \, |s-\bar s|^2 + \frac{1}{\alpha_\varepsilon}\right),
\]
for all $\displaystyle  |s-\bar s| < c$.
\end{proposition}
\begin{proof}
We write
\[
T := T_0 + S,
\]
where we hope that $S$ will be small. Inserting this into (\ref{eq:vkappa}) and using the fact that
\[
\bar c^2 \, e^{-2T_0} = \frac{\kappa^2}{(\cosh (\kappa s/\varepsilon))^2},
\]
we can rewrite the equation we want to solve as
\[
\partial_s^2 S = - K_g \, (T_0 +S)- \frac{(\kappa/\varepsilon)^2}{(\cosh (\kappa \cdot /\varepsilon))^2} \, (1- e^{-2S}).
\]
with $S(\bar s) = \partial_s S(\bar s) =0$. Hence
\begin{equation}
S = \int_{\bar s}^s \left( \int_{\bar s}^t \left( - K_g \, (T_0 +S)- \frac{(\kappa/\varepsilon)^2}{(\cosh (\kappa \cdot /\varepsilon))^2} \, (1- e^{-2S}) \right) \, dt \right) \, ds.
\label{eq:4.5}
\end{equation}
Observe that, on the right-hand-side, the term which does not depend on $S$ can be estimated by
\[
\left| \int_{\bar s}^s \left( \int_{\bar s}^t K_g \, T_0 \, dt \right)\, ds +  K_g (\bar s) \, \ln \kappa \,  \frac{(s-\bar s)^2}{2}  \right| \leq C\, \left( \alpha_\varepsilon \, |s-\bar s|^3  + \frac{|s-\bar s|}{\alpha_\varepsilon}\right),
\]
and its derivative is given by
\[
\left| \int_{\bar s}^s K_g \, T_0 \, dt + K_g (\bar s) \, \ln \kappa \,  \frac{(s-\bar s)^2}{2} \right| \leq C \, \left( \alpha_\varepsilon \, |s-\bar s|^2  + \frac{1}{\alpha_\varepsilon}\right).
\]
Once these estimates have been derived, it is easy exercise the existence of $S$ using a fixed-point argument building on the fact that, in (\ref{eq:4.5}), the linear terms in $S$
\[
\int_{\bar s}^s \left( \int_0^t  - K_g \, S \, dt \right) \,  ds,
\]
and the nonlinear terms
\[
\int_{\bar s}^s \left( \int_0^t \frac{(\kappa/\varepsilon)^2}{(\cosh (\kappa \cdot /\varepsilon))^2} \, (e^{-2S}- 1) \, dt \right) \,  ds,
\]
can be considered as perturbations provided $|s-\bar s| < c$  where $c >0$ is fixed small enough, independently of $\varepsilon$. We leave the details to the interested reader.
\end{proof}

As alluded above, the idea for constructing solutions of the Jacobi-Toda equation (\ref{eq:4.1}) is start with a non-degenerate Bouncing Jacobi Field $\Phi$ with minimums at points of coordinates $s_1 < \ldots < s_n$. We then perturb $\alpha_\varepsilon \, \Phi$ as we have done in Proposition~\ref{pr:4.1}.
This yields a solution $\Psi$ of (\ref{eq:4.1}) which is well defined away from small intervals of size $C/\alpha_\varepsilon$ centered at the $s_j$'s and which depends on the $2n$ parameters $\delta_j$ and $\theta_j$, for $j=1, \ldots, n$.
Next we apply the last Proposition using the parameters  $\kappa = \kappa_j$ and $\bar s = \bar s_j$ to define functions $T_j$ which solve (\ref{eq:4.1}) on some small interval centered at the $s_j$'s.
Our aim is to prove that one can choose the parameters $\delta_j$, $\theta_j$, $\kappa_j$ and $\bar s_j$ in such a way that $\Psi$ and $v_j$ coincide on the intervals where they are both defined.
Since all functions are solution of the same second order differential equation, it is enough to check that one can choose the $4n$ parameters so that the value and derivatives of the different functions agree at one point of the intervals where they are both defined and this will produce the  solution of the Jacobi-Toda equation we are looking for.

\medskip

To be more precise, we have the:
\begin{proposition}
\label{pr:4.3}
Suppose $\Phi$ is a non-degenerate Bouncing Jacobi Field with $n$ minimums at $s_1, \ldots, s_n$. Then, there exists a constant $C >0$ and, for all $\varepsilon >0$ small enough, there exists $\Psi_\varepsilon$ a solution of (\ref{eq:4.1}) obtained by patching together the solutions $T(\varepsilon, \kappa_j, \bar s_j)$ given by Proposition~\ref{pr:4.2} centered at points close to the $s_j$'s and the solution $\Psi (\varepsilon, \delta_j, \theta_j)$ given by Proposition~\ref{pr:4.1} for parameters satisfying
\[
\left|\frac{\kappa_j}{\varepsilon} - \alpha_\varepsilon \, \frac{N_j}{2} \right| +  \alpha_{\varepsilon}^{2}\,  \left|\delta_j + \frac{\ln (N_j/ \bar c)}{\alpha_\varepsilon}\right| + \alpha_{\varepsilon}^2\,  |\theta_j| + \alpha_\varepsilon \, |\bar s_j - s_j| \leq C.
\]
\end{proposition}
\begin{proof}
For all $j=1, \ldots, n$, we define $T_j$ to be the solution of (\ref{eq:4.1}) given by  Proposition~\ref{pr:4.2} when $\kappa = \kappa_j$ and $\bar s = \bar s_j$. This function can be expanded as
\[
T_j (s) = \frac{\kappa_j}{\varepsilon} \, |s-\bar s_j| - \ln \left(\frac{2 \, \kappa_j}{\bar c}\right)  + K_g (\bar s_j) \, \ln \kappa_j \, \frac{(s-\bar s_j)^2}{2} + \mathcal O \left( \alpha_\varepsilon \, (s-\bar s_j)^3\right),
\]
for $C/\alpha_\varepsilon \leq |s-\bar s_j|\leq c$.
\medskip

Next, we expand the solution $\Psi$ obtained in Proposition~\ref{pr:4.1} as
\[
\Psi (s) = \displaystyle \alpha_{\varepsilon} \, \left( 1 + \delta_j + \partial_+ \tilde \Phi (\tilde s_j) \, |s-\tilde s_j| - K_g (\tilde s_j)\, \frac{(s-\tilde  s_j)^2}{2}\right) + \mathcal O (e^{-\alpha_\varepsilon |s-\tilde s_j|/C}),
\]
for $|s - \tilde s_j| \geq  C/\alpha_\varepsilon$.

\medskip

We look for the parameters so that the following system of equation is fulfilled
\begin{equation}
    \label{eq:4.6}
\Psi (s) = T_j(s) \qquad \text{and} \qquad  \partial_s \Psi (s) = \partial_s T_j(s),
\end{equation}
for all $s$ satisfying
\[
\frac{C}{\alpha_\varepsilon} \leq |s-s_j| \leq c.
\]
Using the estimates we have obtained in the previous two Propositions, one can apply some fixed point argument (or the implicit function Theorem) to get the existence of parameters $\kappa_j$, $\bar s_j$, $\delta_j$ and $\theta_j$ for which (\ref{eq:4.6}) holds. The estimates in the statement follow from these equations evaluated at $s-s_j = \alpha_\varepsilon^{-2/3}$.
\end{proof}

\begin{remark}
\label{re:4.4}
In \cite{Manuel2}, the Jacobi-Toda equation
\[
\varepsilon^2 \, {\mathfrak J}_\gamma \Psi  + \bar c^2 \, e^{-2\Psi} =0
\]
also plays a very important role as it does in the present paper. To explain the difference between the two approaches, let us consider the case where the Gauss curvature is constant positive.

\medskip

In this case it is clear that there exists a constant solution of the Jacobi-Toda equation which can be expanded as
\[
\Psi = -\ln \varepsilon - \frac{1}{2} \, \ln \left( - \ln \varepsilon\right) + \mathcal O (1),
\]
In the case where the Gauss curvature is not constant, one can get an approximate solution of the Jacobi-Toda equation by using a fixed point iteration scheme
\[
\Psi_{k+1} = \frac{1}{K_g} \left( \frac{\bar c^2}{\varepsilon^2} \, e^{-2\Psi_k} - \partial_s^2 \Psi_k\right) .
\]
which, for $k$ large enough,  can be expanded as
\[
\Psi_k = -\ln \varepsilon - \frac{1}{2} \, \ln \left( - \ln \varepsilon\right) + \mathcal O (1),
\]
where, this time, $\mathcal O(1)$ depends on $s$.

\medskip

In contract, our solution of the Jacobi Toda equation is close to $\alpha_\varepsilon \, \Phi$ where $\Phi$ is a Bouncing Jacobi Field and where $\alpha_\varepsilon$ is determined by (\ref{eq:4.2}). \end{remark}

\section{The Morse index of solutions to the Jacobi-Toda equation}
In the previous section, starting from a non-degenerate Bouncing Jacobi Field $\Phi$, we have constructed a solution $\Psi_\varepsilon$ of the Jacobi-Toda equation (\ref{eq:4.1}), for all $\varepsilon >0$ small enough.

\medskip

Our aim is now to analyse the mapping properties and the spectrum of $\mathcal L_\varepsilon$, the linearized Jacobi-Toda equation about $\Psi_\varepsilon$
\[
\mathcal L_\varepsilon := \varepsilon^2 \, \mathfrak J_\gamma - 2\, \bar{c}^2 \, e^{-2\Psi_\varepsilon} .
\]
In particular, we are interested in the Morse index of the solution $\Psi_\varepsilon$.

\medskip

Since $\bar{c}^2 e^{-2\Psi_\varepsilon}>0,$ we know that the Morse index of the operator $\mathcal L_\varepsilon$ is bounded from below by the Morse index of the geodesic $\gamma$ which, by definition, is the Morse index of the Jacobi operator $\mathfrak J_\gamma$ also defined on $H^1\left(\gamma \right)$. Therefore, we have
\[
\text{Ind} (\Psi_\varepsilon ) \geq \text{Ind} (\gamma).
\]

We will prove that, for all $\varepsilon >0$ close enough to $0$, the operator $\mathcal L_\varepsilon$ is injective and that it has Morse index equal to $n + \text{Ind} \, (\Phi)$, where $n$ is the number of mimimums of the Bouncing Jacobi Field $\Phi$ we have used to construct $\Psi_\varepsilon$ and $\text{Ind} (\Phi)$ is the Morse index of $\Phi$ as defined in the introduction which agrees with the Morse index of the Hessian of $\mathcal H_n$.
The rational behind this result is that the negative eigenvalues of $\mathcal L_\varepsilon$ arise either from negative directions of  deformation of the critical point with respect to the functional $\mathcal{H}_n$, which in turn correspond to negative eigenvalues of the quadratic form $\mathcal{Q}$ which has been defined in (\ref{eq:3.2}), or they arise from the fact that, close to the points where  $\Phi$ is minimal, $\Psi_\varepsilon$ is close to a solution of the Toda equation whose linearized operator has one negative eigenvalue and, since we have used $n$ copies of a solution of the Toda equation to construct $\Psi_\varepsilon$, this contributes to $n$ additional negative eigenvalues of $\mathcal L_\varepsilon$.

\medskip

As a warm up, let us analyse the spectrum of the operator
\[
L_0:= - \left( \partial_s^2 + \frac{2}{(\cosh s)^2} \right).
\]
which arises as the linearized operator of the Toda equation
\[
\partial_s^2 T - e^{-2T} =0,
\]
at
\[
T (s) := \ln \, (\cosh s).
\]
It can be checked through direct computation that $-1$ is an eigenvalue associated to the eigenfunction
\[
s \mapsto \frac{1}{\cosh s}.
\]
In fact it is known that this is the only negative eigenvalue and hence $L_0$ has index one. The continuous spectrum of $L_0$ is given by $[0, +\infty)$ and the null space of $L_0$ is spanned by the functions
\[
\varphi_0^- (s) := \tanh s \qquad \text{and} \qquad \varphi_0^+ (s) := s \, \tanh s -1.
\]
These functions reflect the fact that the Toda equation is invariant under translation and scaling, namely, if $T$ is a solution of the Toda equation, then $s \mapsto T (\kappa\,  s  + \bar s ) - \ln \kappa$ is also a solution for all $\kappa >0$ and $\bar s \in \R$.

\medskip

The main result of this section reads:
\begin{proposition}
\label{pr:5.1}
For all $\varepsilon >0$ close enough to $0$, the Morse index of the operator $\mathcal L_\varepsilon$ is equal to
\[
n+\text{\rm Ind} (\Phi),
\]
where $n$ is the number of minimums of the non-degenerate Bouncing Jacobi Field $\Phi$ we have used in Proposition~\ref{pr:4.3} to construct $\Psi_\varepsilon$.
\end{proposition}

The proof of this Proposition is decomposed into  two parts. First, we prove that, provided $\varepsilon$ is small enough, the index of $\mathcal L_\varepsilon$ is \emph{at least} equal to $n+\text{\rm Ind} (\Phi)$ and then, we prove that it is at most equal to this value and that $\mathcal L_\varepsilon$ does not have any kernel.

\medskip

We start with the:
\begin{lemma}
\label{le:5.2}
For all $\varepsilon$ small enough, the index of $\mathcal L_\varepsilon$ is \emph{at least} equal to $n+\text{\rm Ind} (\Phi)$.
\end{lemma}
\begin{proof}
The construction in Proposition \ref{pr:4.3} tells that $\Psi_{\varepsilon}$
has local minimums at points $\bar s_{j}$ which depend on $\varepsilon$ but are close to the points $s_j$ where the Bouncing Jacobi Field $\Phi$ is minimal. It also follows from Proposition~\ref{pr:4.3}  and Proposition~\ref{pr:4.1}, that, in a fixed small neighborhood of $\bar s_j$ (which contains $s_j$), the function $\Psi_{\varepsilon}$ can be expanded as
\[
\begin{array}{rllll}
\Psi_\varepsilon (s) & = &  \displaystyle \ln \left( \cosh \left(\frac{\kappa_{j}}{\varepsilon} (s-\bar s_{j})\right)\right) - \ln \kappa_{j} + \ln \bar c + K_g(\bar s) \, \ln \kappa_j \, \frac{(s-\bar s_j)^2}{2} \\[3mm]
& + & \displaystyle \mathcal O \left(\alpha_\varepsilon \, |s-\bar s_j|^3+ \frac{|s-\bar s_j|}{\alpha_\varepsilon}\right),
\end{array}
\]
where $\kappa_j$ and $\bar s_j$ depend on $\varepsilon$.

\medskip

For all nonzero function $v \in H^1(\gamma)$, we define
\[
\mathcal Q (v) : = \displaystyle \int_\gamma \left( \varepsilon^2 \, (\partial_s v)^2 -  \left( \varepsilon^2 \, K_g + 2 \, \bar c^2\, e^{-2\Psi_\varepsilon}\right) \, v^2\right) \, ds.
\]

Given $j =1, \ldots, n$, we define the function $e_j : \gamma \to \R$ by
\[
e_j(s) :=  \frac{\chi (s-\bar s_j) }{\displaystyle \cosh \left(\frac{\kappa_j}{\varepsilon} (s-\bar s_j)\right)},
\]
where $\chi$ is a cutoff function identically equal to $1$ in some fixed small interval $(-c,c)$ and identically equal to $0$ outside $(- 2c, 2c)$. Using the fact that
\[
\left(\varepsilon^2 \, \partial_s^2 + \frac{2\, \kappa_j^2}{\displaystyle \cosh^2 \left(\frac{\kappa_j}{\varepsilon} (\cdot -\bar s_j)\right)}\right) \, \tilde e_j =  \kappa_j^2 \, \tilde e_j,
\]
in $(-c,c)$, together with the estimates in Proposition~\ref{pr:4.1}, one can check that
\[
\left|\mathcal Q (e_j) + \kappa_j^2\, \int_\gamma e_j^2 \, ds \right| \leq \frac{C}{\alpha_\varepsilon^2}.
\]
This estimate proves that, for $\varepsilon$ small enough, the vector space of functions over which $\mathcal Q$ is negative definite is at least $n$ dimensional.

\medskip

Now, assume that we are given a nonzero function $\eta : \gamma -\{s_1, \ldots, s_n\} \to \R$ solution of
\begin{equation}
    \label{JTVAP}
    \left\{
\begin{array}{rlllll}
\mathfrak J_\gamma \eta & = & \lambda \, \eta \quad \text{on} \quad \gamma-\{ s_1, \ldots, s_n\}\\[3mm]
\eta (s_j^+)  + \eta (s_j^-) & = & 0 \quad \text{for all} \quad j=1, \ldots, n\\[3mm]
\displaystyle \partial_s \eta (s_j^+) + \partial_- \eta (s_j^-) + \displaystyle 2 K_g (s_j) \, \frac{\eta (s_j^+) - \eta (s_j^-)}{N_j} & = & 0\quad \text{for all} \quad j=1, \ldots, n.
\end{array}
\right.
\end{equation}
with $\lambda \leq 0$, where we recall that $N_j:= \partial_s\Phi (s_j^+) - \partial_s \Phi (s_j^-)$.
Observe that the function $\eta$ is not necessarily continuous at $s_j$.

\medskip

We set
\[
\varphi_j (s) := \partial_s \Psi_{\varepsilon}(s),
\]
and
\[
\psi_j (s)  := \partial_\kappa  T_{\kappa} \,_{|\kappa = \kappa_j} (s).
\]
where $T_\kappa = T(\varepsilon, \kappa, \bar s_j)$ is the solution given in Proposition~\ref{pr:4.2} for $\kappa$ close to $\kappa_j$ and $\bar s = \bar s_j$ and which agrees with $\Psi_\varepsilon$ when $\kappa = \kappa_j$.
We have
\[
\left( \varepsilon^2 \, (\partial_s^2 + K_g)  + 2 \, \bar c^2 \, e^{-2\Psi_\varepsilon} \right) \, \varphi_j = \varepsilon^2 \, \partial_s K_g \, \Psi_\varepsilon,
\]
and, since $\kappa \mapsto T_\kappa$ is differentiable, we also have
\[
\left( \varepsilon^2 \, (\partial_s^2 + K_g)  + 2 \, \bar c^2 \, e^{-2\Psi_\varepsilon} \right) \, \psi_j =0,
\]
in some small neighborhood of fixed size, centered at $\bar s_j$. These functions can be used to desingularize $\eta$, which as mentioned above, might not even be continuous at the $s_j$'s. More precisely, since the function $\varphi_j$ can be expanded as
\[
\varphi_j (s) = \displaystyle \frac{\kappa_j}{\varepsilon}\, \tanh \left(\frac{\kappa_j}{\varepsilon} (s-\bar s_{j})\right) + K_g(\bar s) \, \ln \kappa_j \, (s-\bar s_j) +  \displaystyle \mathcal O \left(\alpha_\varepsilon \, (s-\bar s_j)^2+ \frac{1}{\alpha_\varepsilon}\right),
\]
and since the function $\psi_j$ can be expanded as
\[
\psi_j (s) = \displaystyle \frac{s-\bar s_j}{\varepsilon}\, \tanh \left(\frac{\kappa_{j}}{\varepsilon} (s-\bar s_{j})\right) \displaystyle + \mathcal O \left(\frac{1}{\kappa_j} + \frac{|s-\bar s_j|^3}{\varepsilon}+ \frac{|s-\bar s_j|}{\kappa_j \alpha_\varepsilon}\right),
\]

We can desingularize $\eta$ in a neighborhood of $s_j$ by defining
\[
\begin{array}{rllll}
\tilde \eta  (s) & : = & \displaystyle (1- \chi (s-\bar s_j)) \, \eta (s) + \chi (s-\bar s_j) \, \frac{\varepsilon}{\kappa_j} \frac{\eta(s_j^+) - \eta (s_j^-)}{2} \,  \varphi_j (s) \\[3mm]
 & + & \displaystyle \chi (s-\bar s_j) \,\, \varepsilon \, \frac{\partial_s \eta(s_j^+) - \partial_s \eta (s_j^-)}{2} \, \psi_j (s),
\end{array}
\]
where $\chi$ is a cutoff function identically equal to $1$ in $(s_j - 1/\alpha_\varepsilon, s_j + 1/\alpha_\varepsilon)$ and identically equal to $0$ outside  $(s_j - 2/\alpha_\varepsilon, s_j + 2/\alpha_\varepsilon)$.

\medskip

Using the fact that
\[
\varepsilon \frac{\ln \kappa_j}{\kappa_j} = - \frac{2}{N_j} + \mathcal O \left(\frac{1}{\alpha_\varepsilon}\right),
\]
we conclude with some work that
\[
\left| \mathcal Q  (\tilde \eta) - \varepsilon^2 \, \lambda \, \int_\gamma \tilde \eta^2 \, ds \right| \leq C \, \frac{\varepsilon^2}{\alpha_\varepsilon}.
\]

This estimate implies that, for $\varepsilon$ small enough, $\mathcal D$ is negative definite over the vector space of functions obtained by desingularizing the solutions of (\ref{JTVAP}) for $\lambda <0$.
We then conclude that the vector space of functions over which $\mathcal Q$ is negative definite is at least $n + \text{Ind} (\Phi)$ dimensional.
\end{proof}

\medskip

We now turn to the second part of the proof. This is the content of the:
\begin{lemma}
\label{le:5.3}
For all $\varepsilon$ small enough, the index of $\mathcal L_\varepsilon$ is at most equal to $n+\text{\rm Ind} (\Phi)$ and $\mathcal L_\varepsilon$ is injective.
\end{lemma}
\begin{proof}
Let us assume that we have a sequence $(\varepsilon_k)_{k=0}^\infty$ tending to $0$ and $\Psi_{\varepsilon_k}$ the corresponding solutions of the Jacobi-Toda equation given by Proposition~\ref{pr:4.3} and $w_k$ a non-zero eigenfunction of $\mathcal L_{\varepsilon_k}$, associated to a non-positive eigenvalue $\lambda_k \leq 0$, namely
\[
\mathcal L_{\varepsilon_k} \, w_k = \lambda_k \, w_k.
\]
Without loss of generality, we can normalize $w_k$ so that $\|w_k\|_{L^\infty}=1$ and choose a point $t_k$ in $\gamma$ where $w_k(t_k)=1$.

\medskip

We now distinguish two cases.

\medskip

\noindent
{\bf Case 1.} Assume that, up to a subsequence, $(t_k)_{k\geq 0}$ converges to a point $t_\infty \notin \{s_1, \ldots, s_n\}$, or that it converges \emph{slowly} to some $s_j$ in the sense that
\[
A := \lim_{k\to +\infty} \frac{1}{\varepsilon_k^2} e^{-2\Psi_{\varepsilon_k} (t_k)}  < +\infty .
\]
In this case, it follows from the maximum principle that
\[
- \sup_\gamma K_g -  2\, \bar c^2 \, A  \leq \limsup_{k\to +\infty} \frac{\lambda_k}{\varepsilon^2_k} \leq 0 ,
\]
and, extracting a subsequence if necessary, we can assume that $(\varepsilon^{-2}_k \, \lambda_k)_{k\geq 0}$ converges to some $\lambda_\infty \leq 0$ and we can also assume that $w_k$ converges to $w_\infty$ which is smooth away from the points $s_1, \ldots, s_n$. Finally, $w_\infty$ is a solution of
\[
\mathfrak J_\gamma    w_\infty = \lambda_\infty w_\infty.
\]
The convergence of $(w_k)_{k=0}^\infty$ to $w_\infty$ is uniform on compacts of $\gamma-\{s_1, \ldots, s_n\}$ and, {\it a priori}, the function $w_\infty$ (which is a solution of a second order linear ordinary differential equation) has both right and left limits and right and left derivatives at $s_1, \ldots, s_n$.

\medskip

Observe that $w_\infty$ is not identically equal to $0$. This is clear when $(t_k)_{k\geq 0}$ converges to $t_\infty \notin \{s_1, \ldots, s_n\}$. If $(t_k)_{k\geq 0}$ converges slowly to $s_j$ as explained above, one can use the fact that $w_k (t_k)=1$ and $w_k$ solves a second order ordinary differential equation with coefficients uniformly bounded to check that $w_k$ remains larger than $1/2$ on some interval $[t_k, t_k+c)$, if $t_k >s_j$ (or $(t_k-c, t_k]$, if $t_k < s_j$). This again implies that $w_\infty$ is not identically equal to $0$ either on the interval $(s_j, s_{j+1})$ (or on the interval  $(s_{j-1}, s_j)$).

\medskip

Now, for each $j=1, \ldots, n$, we fix $\dot s >0$ small and compute
\begin{equation}
\begin{array}{rllll}
\displaystyle \int_{s_j-\dot s}^{s_j+\dot s} \left( \psi \, \partial_s^2 w_k - w_k \, \partial_s^2 \psi \right) ds & = & \left( \psi \, \partial_s w_k - w_k \, \partial_s \psi \right) (s_j+\dot s) \\[3mm]
& - & \left( \psi \, \partial_s w_k - w_k \, \partial_s \psi \right) (s_j-\dot s) ,
\end{array}
\label{eq:5.2}
\end{equation}
where
\[
\psi := \partial_s \Psi_{\varepsilon_k}.
\]
We have
\[
\mathcal L_{\varepsilon_k} \psi = \varepsilon_k^2\, \partial_s K_g \, \Psi_{\varepsilon_k},
\]
hence
\[
\int_{s_j-\dot s}^{s_j+\dot s} \left( \psi \, \partial_s^2 w_k - w_k \, \partial_s^2 \psi \right) ds =  \int_{s_j-\dot s}^{s_j+\dot s} \left( \partial_s K_g \, \Psi_{\varepsilon_k} \, w_k -\frac{\lambda_k}{\varepsilon_k^2} w_k\, \psi \right)\, ds .
\]
Since $\Psi_{\varepsilon_k}+|\partial_s \Psi_{\varepsilon_k}|\leq C \, \alpha_{\varepsilon_k}$ for some constant $C >0$ independent of $k$, we get
\[
\left| \int_{s_j-\dot s}^{s_j+\dot s} \left( \partial_s K_g \, \Psi_{\varepsilon_k} \, w_k -\frac{\lambda_k}{\varepsilon_k^2} w_k\, \psi \right)\, ds \right| \leq C\, \dot s \, \alpha_{\varepsilon_k} .
\]

It follows from the construction of $\Psi_{\varepsilon_k}$ that, away from $s_j$, the function $\Psi_{\varepsilon_k}$ is a small perturbation of $\alpha_{\varepsilon_k}\, \Phi$. Hence,
\[
\lim_{k\to +\infty} \frac{\psi (s_j\pm\dot s)}{\alpha_{\varepsilon_k}}= \partial_s \Phi(s_j+\dot s) ,
\]
and
\[
\lim_{k\to +\infty} \frac{\partial_s \psi (s_j\pm \dot s)}{\alpha_{\varepsilon_k}}= \partial^2_s \Phi (s_j\pm \dot s) =  K_g(s_j\pm \dot s) \, \Phi(s_j\pm \dot s).
\]
Letting $k$ tend to infinity in (\ref{eq:5.2}), we conclude that
\[
\left| (\partial_s \Phi \, \partial_s w_\infty - K_g w_\infty ) (s_j+\dot s) - (\partial_s \Phi \, \partial_s w_\infty - K_g w_\infty ) (s_j-\dot s) \right| \leq C\, \dot s.
\]
Letting $\dot s$ tend to $0$, we conclude that
\begin{equation}
(\partial_+ w_\infty + \partial_- w_\infty) (s_j) +  \frac{2K_g (s_j)}{N_j} \, (w_\infty(s_j^+) - w_\infty(s_j^-)) = 0.
\label{eq:5.3}
\end{equation}

Performing the same analysis with
\[
\psi (s) := \partial_\kappa T_{\kappa} |_{\kappa = \kappa_j} (s),
\]
we obtain
\[
w_\infty (s_j^+)= -w_\infty (s_j^-).
\]
which together with (\ref{eq:5.3}) implies that  $w_\infty$ is continuous on $\gamma$ away from $s_j$, solves $\mathfrak J_\gamma    w_\infty = -\lambda_\infty w_\infty(s^+_j)$ on $(s_j, s_{j+1})$ and
\[
\partial_s w_\infty (s_j^+) + \partial_s w_\infty (s_j^-) + \frac{4 K_g}{N_j} \, w_\infty =0,
\]
at $s_j$, for all $j=1, \ldots, n$. Since we have assumed the Bouncing Jacobi Field to be non-degenerate, we also conclude that $\lambda_\infty \neq 0$.

\medskip

\noindent
{\bf Case 2.} Assume that (up to a subsequence), the sequence $(t_k)_{k=0}^\infty$ converges \emph{fast} to some $s_j$, in the sense that
\[
\lim_{k\to +\infty} \frac{1}{\varepsilon_k^2}\, e^{-2 \Psi_{\varepsilon_k} (t_k)} = +\infty
\]
In this case, we define
\[
\hat w_k (s):= w_k\left( s_j + \frac{\varepsilon_k}{\kappa_{k,j}} s \right),
\]
which is a solution of
\[
\left( \partial^2_s + \left(\frac{\varepsilon_k}{\kappa_{k,j}}\right)^2 \, \hat K_g \right) \, \hat w_k +2 \, \left( \frac{\bar c}{\kappa_{k,j}}\right)^2 \, e^{-2\hat \Psi_k} \hat w_k=- \frac{\lambda_k}{\kappa_{k,j}^2} \, \hat w_k,
\]
where
\[
\hat K_g(s) := K_g \, \left( s_j + \frac{\varepsilon_k}{\kappa_{k,j}} s \right)\qquad \text{and}\qquad  \hat \Psi_k (s) := \Psi_{\varepsilon_k}\left( s_j + \frac{\varepsilon_k}{\kappa_{k,j}} s \right).
\]
Since we have assumed that the $L^\infty$ norm of $w_k$ tends to zero on any compact interval away from the local minimums, there exists a constant $c < 0$ independent of $\varepsilon$, such that
\[
\frac{\lambda_k}{\kappa^2_{k,j}}< c.
\]
The maximum principle also tells that for some  constant $C >0$ independent of $\varepsilon$,
\[
\frac{\lambda_k}{\kappa^2_{k,j}} > - C.
\]
Therefore we can assume that up to a subsequence,
\[
\lim_{k\to \infty}\frac{\lambda_k}{\kappa^2_{k,j}}= \lambda_\infty.
\]
Then the sequence of functions $\hat w_k$ will converge to $w_\infty$ which satisfies
\[
-  (\partial^2_s + 2 \bar c^2 e^{-2 u}) \, w_\infty = \lambda_\infty w_\infty.
\]
We claim that $w_\infty$ is not identically equal to $0$. Assuming the claim is proven, this implies that that $w_\infty$ is an eigenfunction of the linearized Toda equation with negative eigenvalue, which is unique as we mentioned before. In particular, this implies that $\lambda_\infty =-1$.

\medskip

It remains to prove the claim. Obviously, the claim holds if (for a subsequence),
\[
\lim_{k\to +\infty} \frac{\kappa_{k,j}}{\varepsilon_k}\, (t_k-s_j) \in \R.
\]
Now, if, for example,
\[
\lim_{k\to +\infty} \frac{\kappa_{k,j}}{\varepsilon_k}\, (t_k-s_j) =+\infty,
\]
we can write
\[
\partial_s^2 \hat w_k \geq - \left(\frac{\varepsilon_k}{\kappa_{k,j}}\right)^2\, \hat K_g  - 2 \, \left( \frac{\bar c}{\kappa_{k,j}}\right)^2 \, e^{-2\hat \Psi_k} =: f_k.
\]
since $\lambda_k \leq 0$ and $\hat w_k\leq 1$. We set
\[
\hat t_k :=\frac{\kappa_{k,j}}{\varepsilon_k}\, (t_k-s_j).
\]
Using the fact that $\hat w_k (\hat t_k) = 1$, $\partial_s \hat w_k (\hat t_k) =0$, we get
\[
\begin{array}{rllll}
\hat w_k (s) & \geq & \displaystyle 1 - \int_s^{\hat t_k} \int_{x}^{\hat t_k} f_k(y) \, dy\, dx \\[3mm]
& \geq & \displaystyle 1 - C \,\left(  \left(\frac{\varepsilon_k}{\kappa_{k,j}}\right)^2 \, (\hat t_k -s)^2 - e^{- 2 \, (\hat t_k-s)}\right),
\end{array}
\]
for all $s \in (0, \hat t_k)$. Observe that
\[
\lim_{k\to +\infty} \frac{\varepsilon_k}{\kappa_{k,j}}\, \hat t_k =0,
\]
hence, we conclude that, if $\hat s >0$ is fixed large enough, then $\hat w_k (\hat s) \geq 1/2$ for all $k$ large enough. This implies that $w_\infty $ is not identically equal to $0$ and the proof of the claim is complete.

\medskip

Combining the discussion for all cases, we see that an eigenfunction of $\mathcal{L}_\varepsilon$ either converges to an eigenfunction of the linearized problem associated to the quadratic form $\mathcal Q$ (after a suitable rescaling), or to an eigenfunction of the linearized Toda equation. The conclusion of the lemma then follows from Proposition \ref{pr:3.4} and Lemma~\ref{le:5.2}.
\end{proof}

\begin{proof}[Proof of Proposition~\ref{pr:5.1}]
Collecting the results of Lemma~\ref{le:5.2} and Lemma~\ref{le:5.3}, we see that the proof of Proposition~\ref{pr:5.1} is now  complete.
\end{proof}



Observe that, as a by product of the proof of Lemma~\ref{le:5.3}, we have the:
\begin{proposition}
\label{pr:5.4}
There exists a constant $c >0$ such that, for all $\varepsilon$ small enough, there are no eigenvalues of $\mathcal L_\varepsilon$ in  $[-c\, \varepsilon^2, c\, \varepsilon^2]$.
\end{proposition}
\begin{proof}
The proof is by contradiction. We assume that there was a sequence of eigenvalues $\lambda_k$ of $\mathcal L_{\varepsilon_k}$ such that
\[
\lim_{k\to +\infty} \frac{\lambda_k}{\varepsilon_k^2} =0.
\]
and $\lim_{k\to +\infty} \varepsilon_k =0$.
Then, the proof of Lemma~\ref{le:5.3} applies verbatim to get a contradiction.
\end{proof}


\begin{remark}
\label{re:5.5}
We have explained in Remark~\ref{re:4.4} that the construction in \cite{Manuel2} was making use of an approximate solution of the Jacobi-Toda equation which is close to a constant which can be expanded as
\[
\psi = -\ln \varepsilon - \frac{1}{2} \, \ln \left( - \ln \varepsilon\right) + \mathcal O (1),
\]
The linearized Jacobi-Toda equation about this solution reads
\[
-  \varepsilon^2 \, \left(\partial_s^2 + K_g \right) + 2 \, \bar c^2\, e^{-2\psi},
\]
but
\[
2 \, \bar c^2\, e^{-2\psi} = - \varepsilon^2 \, \left( K_g \ln \varepsilon + \mathcal O \left(\ln \, (-\ln \varepsilon) \right)\right),
\]
and this implies that the Morse index of this linearized equation is of order $-\ln \varepsilon$ and hence tends to infinity as $\varepsilon$ tends to $0$. This in turn implies that the solutions of the Allen-Cahn equation constructed in \cite{Manuel2} have Morse indices which tend to infinity as $\varepsilon$ tends to $0$.
\end{remark}

Since the index of $\mathcal L_{\varepsilon}$ is bounded from below by the index of geodesic
$\gamma,$ we have the following:
\begin{corollary}
\label{co:5.6}
Suppose there is a non-degenerate Bouncing Jacobi Field $\Phi$ with $n$ minimums for the geodesic $\gamma$. Then $2n\geq \text{\rm Ind}\left(  \gamma\right).$
\end{corollary}
\begin{proof}
We can construct a solution $\Psi_\varepsilon$ of the Jacobi-Toda equation using the Bouncing Jacobi Field $\Phi$. Now we know that the corresponding index of $\mathcal{L}_\varepsilon$ is equal to $
n + \text{\rm Ind} (\Phi)$. From this, we deduce that
\[
n+\text{\rm Ind} (\Phi)\geq \text{\rm Ind} (\gamma).
\]
Since $\text{\rm Ind} (\Phi)\leq n$, we conclude that
\[
2n\geq \text{\rm Ind}(\gamma),
\]
which is the required inequality.
\end{proof}

This last result provides a necessary condition for the existence of  Bouncing  Jacobi Fields with prescribed number of minimums.
\section{Local geometry about a closed geodesic and the existence of solutions to the Allen-Cahn equation}

In this section, we explain how the previous analysis leads to solutions of the Allen-Cahn equation.
Observe that, away from the transition layers, the solutions of the Allen-Cahn equation converge exponentially fast to either $1$ or $-1$ (in fact the convergence is like $\exp (-\text{dist} (\cdot, \gamma)/(\sqrt 2 \, \varepsilon))$. This implies that the construction near a given geodesic is hardly perturbed by the construction near a different geodesic. The construction of solutions whose transition layer converges with multiplicity one to a geodesic is the same as the one done in \cite{Pacard, Pac-Role}).
Hence, to simplify the exposition, we will assume ${\bf G}^1$ is empty.
The construction of solutions whose transition layer converges with multiplicity two to a geodesic follows the same lines as in \cite{Manuel2} with some differences we will explain. To simplify the exposition, we will assume that ${\bf G}^2$ is reduced to a single geodesic. Thanks to the exponential convergence away from the transition layers, the construction in the general case follows the same line with more involved notations.

\medskip

Assume we are given a two-sided, closed embedded geodesic $\gamma$ in $(M,g)$. We choose a unit normal vector field $\nu$ along $\gamma$ and we define
\[
Z (s,t) := \exp_{s}\left(t\,\nu (s)\right),
\]
where $s \in \gamma$ and $t\in \R$ is close to $0$. It is well known that $Z$ is a diffeomorphism from $\gamma \times (-\tau, \tau)$ into a tubular neighborhood of $\gamma$ in $M$, provided $\tau >0$ is fixed small enough. The coordinates $(s,t) \in \gamma\times (-\tau,\tau)$ are called \emph{Fermi coordinates} relative to the geodesic $\gamma$. Observe that $t$ is \emph{signed distance} to the geodesic $\gamma$.

\medskip

In these coordinates, the metric on $M$ can be expanded as
\[
g = \left( 1 - K_g (s) \, t^2 +\mathcal O (t^3) \right) \, ds^2 + dt^2 ,
\]
where $s \mapsto K_g(s)$ is the Gauss curvature along $\gamma$ \cite{Pacard, Pac-Role}. The fact that, in this expansion, the coefficient of $ds^2$ does not involve any linear term in $t$ is a consequence of the fact that $\gamma$ is a geodesic, and the fact that there are no terms in $ds\,dt$ in the expression of $g$ in these coordinates is a consequence of Gauss's Lemma.

\medskip

In these coordinates, the Laplace-Beltrami operator on $M$ is given by
\[
\Delta_{g} = \displaystyle \frac{1}{\sqrt{ |g|}} \partial_s \left( \sqrt{|g|} \, g^{ss} \, \partial_s\right) + \frac{1}{\sqrt {|g|}} \partial_t \left( \sqrt{|g|} \, g^{tt} \, \partial_t\right)
\]
where $|g|$ is the coefficient of $ds^2$ in the expression of $g$ and $g^{ss}$ is the inverse of this coefficient. Moreover, $g^{tt} =1$. One can check that $\Delta_{g}$ has the following expansion close to $\gamma$
\begin{equation}
\label{lap-Bel}
\Delta_{g} = (1+ \mathcal O(t^2)) \, \partial_s^2 + {\mathcal O} (t^2) \, \partial_s  +  \partial_t^2  - \left(K_g(s) + \mathcal O (t) \right) \, t \,\partial_t.
\end{equation}

 \medskip

Recall that the \emph{heteroclinic solution} of the Allen-Cahn equation is defined to be the unique entire solution of
\[
W^{\prime\prime}+W-W^3=0
\]
satisfying $\displaystyle \lim_{x\to \pm \infty} W(x) = \pm 1$ and $W(0)=0$. It is explicitly given by
\[
W(x) := \tanh \left( \frac{x}{\sqrt 2}\right).
\]
To build a solution of the Allen-Cahn equation on $M$ whose transition layer converges, with multiplicity $1$, to a given closed embedded geodesic, it is reasonable to look for a (small) perturbation of the function explicitly given in Fermi coordinates by
\[
U_\varepsilon (s,t) := W \left(t/\varepsilon \right) .
\]
To allay notations, we set
\[
\mathcal N_\varepsilon (U) := \varepsilon^2 \, \Delta_g U + U - U^3.
\]
Using the expansion of the Laplace-Beltrami operator in Fermi coordinates, we get the expression of the error
\[
\displaystyle {\mathcal N}_\varepsilon ( U_\varepsilon) (s,t) = - \varepsilon^2 \, K_g (s) \, \frac{t}{\varepsilon} \,  W'\left(t/\varepsilon\right)\, (1+ \mathcal O (t)) .
\]
Observe that, the function on the right-hand-side is bounded by a constant $\varepsilon^2$ and decays exponentially fast as one moves away from $\gamma$.

\medskip

In the normal direction to the geodesic, the approximate solution is, up to a scaling factor, close to a translated copy of the one dimensional heteroclinic solution $W$. Since the one dimensional Allen-Cahn equation is translation invariant, the linearized Allen-Cahn operator about $W$ has a kernel spanned by
\[
W^{\prime} (x) = \displaystyle \frac{1}{\sqrt 2} \, \frac{1}{\cosh^2 \left( x/\sqrt 2 \right)}.
\]
That is
\[
\left(W^{\prime}\right)^{\prime\prime} + ( 1 - 3W^{2}) \, W^{\prime}=0.
\]
This implies that
\[
\varepsilon^2 \, \Delta_g + 1- 3 \, U_\varepsilon^2,
\]
the linearized Allen-Cahn operator about $U_\varepsilon$ might not be invertible but even if it were, it would have a right inverse whose norm will blow up as the parameter $\varepsilon$ tends to $0$.
To overcome this issue, the classical strategy \cite{Pacard, Pac-Role} is to work orthogonally to an approximate kernel which turns out to be the space of functions which can be written as
\[
(s,t) \mapsto \varphi(s) \, W^\prime (t/\varepsilon),
\]
for any $\varphi \in H^1(\gamma)$ in some tubular neighbourhood of $\gamma$.

\medskip

The price to pay is that we have to modify the approximate solution by introducing perturbations of its zero set. Given a (small) function $f : \gamma \to \R$ at least $\mathcal C^2$, we are led to alter the previous approximate solution and consider instead the function
\[
U_\varepsilon (s,t) := W\left( \frac{t - \varepsilon \, f(s)}{\varepsilon}\right),
\]
as a reasonable approximate solution. Indeed, using (\ref{lap-Bel}), we get the structure of the error
\begin{equation}
    \label{eq:6.2}
\begin{array}{rlll}
\displaystyle {\mathcal N}_\varepsilon (U_\varepsilon ) (s,t)
& = & \displaystyle  \varepsilon^2 \, \left( {\mathfrak J}_\gamma f (s) - K_g(s) \, \left(\frac{t - \varepsilon \, f(s)}{\varepsilon}\right)\right) \, W'\left( \frac{t - \varepsilon \, f(s)}{\varepsilon}\right) \\[3mm]
& + & \displaystyle\varepsilon^2 \,  (\partial_s f)^2 \, W'' \left( \frac{t - \varepsilon \, f(s)}{\varepsilon}\right)\\[3mm]
& + & \displaystyle\varepsilon^2 \, \left( \mathcal O \left(\frac{t^2}{\varepsilon} \right) + \mathcal O (t^2) \, \partial_s f + \mathcal O (t^2) \, \partial^2_s f\right) W' \left( \frac{t - \varepsilon \, f(s)}{\varepsilon}\right)\\[3mm]
& + & \displaystyle\varepsilon^2 \, \mathcal O (t^2) \,  (\partial_s f)^2 \, W'' \left( \frac{t - \varepsilon \, f(s)}{\varepsilon}\right).
\end{array}
\end{equation}
Observe that if we fix $c >0$ small enough and if we require that $\mathcal N_\varepsilon (U_\varepsilon)$ is $L^2(-c,c)$ orthogonal to
\[
t\mapsto  W'\left(\frac{t - \varepsilon \, f(s)}{\varepsilon}\right),
\]
we find that $f$ has to  solve an equation of the form
\[
{\mathfrak J}_\gamma f  = \mathcal O (\varepsilon ) + \varepsilon \, L (f) + \varepsilon \, Q (f) ,
\]
where $L$ is a linear second order differential operator and $Q$ a quadratic second order differential operator in the sense that its Taylor expansion in $f$ vanishes at order $2$.
To derive this equation it is important to use the fact that $x\mapsto x\, W'(x)$ and $x\mapsto W''(x)$ are both $L^2(\R)$ orthogonal to $x\mapsto W'(x)$. The fact that we work on some interval $(-c,c)$ instead of $\R$ is not a real issue since the function $W'$ tends exponentially fast to $0$ at infinity.

\medskip

Therefore, assuming that ${\mathfrak J}_\gamma$ is invertible, one can improve the approximate solution by choosing $f$ to be a solution of the above equation (which can then be solved using a fixed point argument) and this explains why we ask the geodesic $\gamma$ to be non-degenerate.
The existence of solutions of the Allen-Cahn equation having a transition layer which converges with multiplicity one is well known \cite{Pac-Role}, therefore, we will not copy this construction here.
However, it is interesting to compare this construction with the construction of the present paper.
In particular, the above formal computation (which can be made rigorous) governs the size of the function $f$ and the solution of the Allen-Cahn equation can be decomposed in a tubular neighborhood of $\gamma$ as
\[
u_\varepsilon = W\left( \frac{t-\varepsilon \, f}{\varepsilon}\right) + V ,
\]
where $\|V\|_{L^\infty} \leq C\,  \varepsilon^2$ and $\|f\|_{L^\infty} \leq C\, \varepsilon$. As we will see shortly the situation is totally different in the construction of solutions whose transition layers converges with multiplicity $2$ to a given geodesic.

\medskip

Let us now explain how to build solutions of the Allen-Cahn equation whose transition layers converge to a given geodesic with multiplicity $2$. We have in mind to look for solutions whose zero set consists of two components (we assume here that the geodesic is two-sided), each one being a normal graph over $\gamma$ for some small function. This time, a reasonable approximate solution $U_\varepsilon$ has the form
\[
U_\varepsilon (s,t) := W \left(\frac{t-\varepsilon \, f_2(s)}{\varepsilon}\right) - W\left(\frac{t- \varepsilon \,f_1(s)}{\varepsilon}\right) + 1,
\]
where $f_{1} < 0 < f_{2}$. More precisely, we will prove that $\sqrt 2\, f_2$ is close to $\Psi_\varepsilon$ and $\sqrt 2\, f_1$ is close to $- \Psi_\varepsilon$ where $\Psi_\varepsilon$ is a solution of the Jacobi-Toda equation. Hence, the distance between the graph of $\varepsilon\, f_2$ and the graph of $\varepsilon \, f_1$ is of order $\varepsilon\, \alpha_\varepsilon$ and hence is much larger than $\varepsilon$, which is the characteristic size of the transition layers of the solutions of the Allen-Cahn equation. Even though the distance between the two transition layers is much larger than $\varepsilon$, we will see that there will be some interaction between the two transition layers at points where $\Psi_\varepsilon$, the solution of the Jacobi-Toda equation, is minimal.

\medskip

To simplify the notation, let us set
\[
W_j (s,t) := W\left(\frac{t-\varepsilon \, f_{j}(s)}{\varepsilon}\right)  \qquad \text{and} \qquad  W_{j}^{\prime}(s,t) := W^{\prime}\left(  \frac{t-\varepsilon \, f_{j}(s)}{\varepsilon}\right)
\]
with similar notations for higher derivatives.
\medskip

Given a small function $V$, we compute
\begin{equation}
\label{eq:sold}
\begin{array}{rlllll}
\mathcal N_\varepsilon (U_\varepsilon + V) & = & \mathcal N_\varepsilon (W_2) - \mathcal N_\varepsilon (W_1) \\[3mm]
& + & 3 \, (1-W_2^2) \, (1 -W _1) - 3\, (W_2+1) \, (1- W_1)^2\\[3mm]
& + & \varepsilon^2 \, \Delta_g V + V - 3 \, U^2_{\varepsilon} \, V\\[3mm]
& - & 3\,  U_\varepsilon \, V^2 - V^3.
\end{array}
\end{equation}
The expression of $\mathcal N_\varepsilon (W_2)$ and $\mathcal N_\varepsilon (W_1)$ are given by (\ref{eq:6.2}) with $f$ replaced respectively by $f_2$ or $f_1$.

\medskip

As in the case of a single transition layer, the projection of the error $\mathcal N_\varepsilon(U_\varepsilon)$ over the approximate kernel plays an important role in our analysis. This time, the approximate kernel is spanned by the functions
\[
(s,t) \mapsto \varphi_j (s) \, W^{\prime}\left(  \frac{t-\varepsilon \, f_j(s)}{\varepsilon}\right)
\]
for $j=1,2$, where the functions $\varphi_j \in H^1(\gamma)$.

\medskip
For example, the projection of $\mathcal N_\varepsilon$ over $W'_2$ reads
\begin{equation}
\label{eq:6.33}
\begin{array}{rllll}
\displaystyle \int_{[- c , c ]} \mathcal N_\varepsilon (U_\varepsilon) \, W_2' \, dt & = & \displaystyle \varepsilon \, \left( c_{0}\, \varepsilon^2 \, \mathfrak J_\gamma f_2 + c_{1} \, e^{-\sqrt{2} (f_2-f_1)} \right) \\[3mm]
& + & \varepsilon^4 \, \mathcal O (1) + \varepsilon^3 \,  e^{-\sqrt{2} (f_2-f_1)} \, \mathcal O (1) + e^{-2 \sqrt{2} (f_2-f_1)} \, \mathcal O (1) \\[3mm]
& + & \varepsilon^4 \,  L(f_2) +\varepsilon^3 \,  e^{-\sqrt{2} (f_2-f_1)} \, L(f_1)  +  \varepsilon^4 \, Q(f_1, f_2)
,
\end{array}
\end{equation}
where $L$ are second order differential operators and $Q$ a quadratic second order differential operator and where the constants $c_0$ and $c_1$ are explicitly given by
\begin{equation}
c_{0}: =\int_{{\R}}(W^\prime(x))^2 \, dx \qquad \text{and} \qquad c_{1}: =6\int_{{\R}}e^{-\sqrt{2}x} \, \left(1-W^{2}(x)\right) W^{\prime}(x)\, dx.
\label{eq:c0c1}
\end{equation}
The term $\varepsilon^2 \, c_0 \, \, \mathfrak J_\gamma f$ comes from the first term in (\ref{eq:6.2}) while the term $c_{1} \, e^{-\sqrt{2} (f_2-f_1)}$ comes from the $L^2$ projection of $3 \, (1-W_2^2) \, (1 -W _1)$ over $W_2'$ using the fact that
\[
1- W (x) = 2 \, e^{- \sqrt{2} x} + \mathcal O \left(e^{- 2 \sqrt{2} x}\right)
\]
for $x >0$. A similar analysis can be performed for the projection of $\mathcal N_\varepsilon$ over $W'_1$. The result is similar with a change of sign due to the fact that $3 \, (1-W_2^2) \, (1 -W _1)$ is now replaced by $- 3 \, (1-W_1^2) \, (1 + W _2)$.

\medskip

Neglecting the higher order terms, we conclude that $U_\varepsilon$ could be a reasonable approximate solution if we are able to find solutions of the system \[\varepsilon^2 \,  c_0 \, \mathfrak J_\gamma \, f_2 + c_1 \, e^{- \sqrt{2} (f_2-f_1)} = 0 \qquad \text{and} \qquad
\varepsilon^2 \, c_0 \, \mathfrak J_\gamma \, f_1 - c_1 \, e^{- \sqrt{2} \, (f_2-f_1)} = 0.
\]
If we set
\[
\phi_{j} :=  \sqrt{2} \, f_j,
\]
for $j=1, 2$ and $\bar c^2 := \frac{\sqrt{2}c_{1}}{c_{0}}$, the above system becomes
\[
\left\{
\begin{array}{rllll}
\varepsilon^2 \, \mathfrak J_\gamma \, \phi_2 + \bar{c}^2 \, e^{- (\phi_2 - \phi_1)} & =& 0 \\[3mm]
\varepsilon^2 \, \mathfrak J_\gamma \, \phi_1 - \bar{c}^2 \, e^{- (\phi_2-\phi_1)} & = & 0,
\end{array}
\right.
\]
Summing up the two equations, we obtain
\[
\mathfrak J_\gamma \left(\phi_{1}+\phi_{2}\right) =0,
\]
and, if the geodesic $\gamma$ is assumed to be non-degenerate, we conclude that $\phi_2 =- \phi_1$. Therefore, our problem reduces to the Jacobi-Toda equation (\ref{eq:4.1}) which has been studied in Section 4.

\medskip

We now come to the main result of this section. Starting from a non-degenerate embedded, two-sided geodesic $\gamma$ and a non-degenerate Bouncing Jacobi Field $\Phi$ along $\gamma$, we prove the existence of solutions to the Allen-Cahn equation whose nodal set is close to the normal graph over $\gamma$ for the functions $s \mapsto \varepsilon \, \Psi_\varepsilon (s)$, where $\Psi_\varepsilon$ is the solution of the Jacobi-Toda equation associated to $\Phi$, given by Proposition~\ref{pr:4.3}.

\medskip

We keep the notations already used and set
\[
\sqrt 2 \, f_2 : = \Psi_\varepsilon + h_2 \qquad \text{and} \qquad \sqrt 2 \,  f_1 : = - \Psi_\varepsilon + h_1,
\]
where $\Psi_\varepsilon$ is the solution of the Jacobi-Toda equation provided by Proposition~\ref{pr:4.3} starting with the Bouncing Jacobi Field $\Phi$ and where $h_{2}$ and $h_{1}$ are small functions to be determined.

\medskip

We define the approximate solution as
\[
U_\varepsilon (s,t) = \chi (t) \, \left( W \left( \frac{t - \varepsilon \, f_2(s)}{\varepsilon} \right) - W \left( \frac{t -\varepsilon \, f_1 (s)}{\varepsilon} \right) \right) + 1 ,
\]
where $\chi$ is a smooth cutoff function identically equal to $1$ in a tubular neighbourhood of width $c$ around $\gamma$ and  identically equal to $0$ outside a tubular neighbourhood of width $2 \, c$ around $\gamma$, for some $c >0$ fixed small enough but independent of $\varepsilon$.

\medskip

We look for a solution of the Allen-Cahn equation of the form $U_\varepsilon + V$ where $V$ is a solution of
\begin{equation}
\left( \varepsilon^{2}\, \Delta_g  + 1 - 3 U_\varepsilon^2 \right)  \, V + \mathcal N_\varepsilon ( U_\varepsilon) - 3 \, U_\varepsilon \,  V^2 - V^3 =0,
\label{linearfi}
\end{equation}
which satisfies the orthogonality conditions
\begin{equation}
    \label{eq:6.55}
    \int_{(-c,c)} V \, W_1' \, dt = \int_{(-c,c)} V \, W_2' \, dt = 0,
\end{equation}
for all $s \in \gamma$.

\medskip

The next statement involves some weighted norm on the space of $\mathcal C^{2, \alpha}$ function which is defined by
\[
\| V \|_{\mathcal C^{k,\alpha}_\varepsilon} =\sup_i \left( \sum_{j=0}^k \varepsilon^j \, \sup_{\Omega_i} \, |\nabla^j V| + \varepsilon^{k+\alpha} \, \sup_{p\neq q \in \Omega_i} \frac{|\nabla^k V(p) - \nabla^k V(q)|}{|p-q|^{\alpha}}\right),
\]
where $(\Omega_i)_{i \in I}$ is a finite open cover of $M$. In other words, with this definition, in a local chart, the $L^\infty$ norm of any $j$-th partial derivatives of $f$ is controlled by $\varepsilon^{-j} \, \| f \|_{\mathcal C^{k,\alpha}_\varepsilon}$.

\medskip

The main result of this section is the:
\begin{proposition}
\label{pr:6.2}
Let $\gamma$ be a closed, embedded, non-degenerate two-sided geodesic in $(M,g)$. Assume that $\Phi$ is a non-degenerate \emph{Bouncing Jacobi Field} with $n \geq 1$ minimums. Then, for all $\varepsilon$ small enough, the Allen-Cahn equation has a solution $u_{\varepsilon}$ whose transition layers converge to $\gamma$ with multiplicity two. Moreover, $u_\varepsilon = U_\varepsilon + V$, where
\[
\| V \|_{\mathcal C^{2,\alpha}_\varepsilon}\leq C \, \varepsilon^2\, (\ln \varepsilon)^2
\]
and
\[
\| h_1 \|_{\mathcal C^{2,\alpha}} + \| h_2 \|_{\mathcal C^{2,\alpha}} \leq C \, \varepsilon \, (-\ln \varepsilon)^5.
\]
\end{proposition}
\begin{proof}
We need to find $V : M \to \R$ and the functions $h_1 : \gamma\in \R$ and $h_2: \gamma \to \R$ such that (\ref{linearfi}) and (\ref{eq:6.55}) hold.

\medskip

As explained above, the linear operator
\[
-\left(\varepsilon^{2} \, \Delta_{g} + 1 - 3\, U_{\varepsilon}^2\right)
\]
can be inverted up to functions of the form $(s,t) \mapsto \ell_1 (s) \, \chi (t) \, W_{1}^{\prime}(t)$ and $(s,t) \mapsto \ell_2 (s)\, \chi (t) \, W_{2}^{\prime}(t)$ where $\chi$ is a cutoff function identically equal to $1$ in a  tubular neighborhood of width $c$ about $\gamma$ and identically equal to $0$ outside a tubular neighborhood of width $2c$ about $\gamma$. We refer to \cite{Manuel1, Pacard, Pac-Role} and especially \cite{M0} for more details.

\medskip

This leads to a coupled system of equations
\begin{equation}
\label{eq:6.77}
\begin{array}{rlll}
-\left(\varepsilon^{2} \, \Delta_{g} + 1 - 3\, U_{\varepsilon}^2\right) V  & =& \mathcal{N}_{\varepsilon}\left(W_2\right) -\mathcal{N}_{\varepsilon}\left(W_1\right)  \\[3mm]
& + & 3\, \left(1-W_2^{2}\right)\left(1-W_1\right)  -3\, \left( 1+W_2\right)  \left(1-W_1\right)^{2}\\[3mm]
& - & 3\, U_{\varepsilon}\, V^{2}-V^{3} + \chi \, \left( \ell_2 \, W'_2 + \ell_1\, W'_1\right),
\end{array}
\end{equation}
and
\begin{equation}
\label{eq:6.78}
\left\{
\begin{array}{rllll}
\varepsilon^2 \, \mathfrak J_\gamma h_2 - \bar c^2\,  e^{-2 \Psi_\varepsilon}  \left(
h_2-h_1\right) & =  & \mathcal O (\varepsilon^3\, (\ln \varepsilon)^4) \, + \varepsilon^3  \, (\ln \varepsilon)^4 \, L (h_1,h_2)\\[3mm]
& + & \displaystyle  \varepsilon \, L(V) + Q(h_1,h_2) + \frac{1}{\varepsilon} \, Q(V),\\[3mm]
\varepsilon^2 \, \mathfrak J_\gamma h_1 + \bar c^2 e^{-2\Psi_\varepsilon} \left(
h_2-h_1\right) & = & \mathcal O (\varepsilon^3\, (\ln \varepsilon)^4) + \varepsilon^3 \, (\ln \varepsilon)^4 \, L (h_1,h_2) \\[3mm]
& + & \displaystyle  \varepsilon \, L(V) + Q(h_1,h_2) + \frac{1}{\varepsilon} \, Q(V).\\[3mm]
\end{array}
\right.
\end{equation}

If we define $\xi := h_{1}-h_{2}$ and $\zeta := h_{1} +h_{2}$, we see that the second system of equation translates into
\[
\left\{
\begin{array}{rllll}
\mathcal L_\varepsilon \xi & = & \mathcal O (\varepsilon^3 \, (\ln\varepsilon)^4) + \varepsilon^3 \, (\ln\varepsilon)^4 \, L(\xi , \zeta)+ Q(\xi, \zeta)\\[3mm]
& + & \displaystyle  \varepsilon \, L(V)  + \frac{1}{\varepsilon} \, Q(V)\\[3mm]
\mathfrak J_\gamma \zeta  & = &  \displaystyle \mathcal O (\varepsilon \, (\ln \varepsilon)^4) + \varepsilon \, (\ln\varepsilon)^4  \, L(\xi , \zeta)+ \frac{1}{\varepsilon^2} \, Q(\xi, \zeta)\\[3mm]
& + & \displaystyle  \frac{1}{\varepsilon} \, L(V) +  \frac{1}{\varepsilon^3} \, Q(V).
\end{array}
\right.
\]

Let us now point out the main ingredients in the estimate for $V$. We need to estimate the projection of the right hand side of (\ref{eq:6.77}) on the orthogonal complement of $\chi \, W'_1$ and $\chi \, W'_2$ when $V=0$.
Due to the exponential convergence of $W$ to $\pm1$, it is clear that the main contribution will come from the estimate of this quantity in a tubular neighborhood about $\gamma$.

\medskip

A close inspection of (\ref{eq:6.2}) shows that, in a tubular neighbourhood of the curve determined by the graph of $\varepsilon \, \Psi_\varepsilon$, the terms in $\mathcal N_\varepsilon (W'_2) - \mathcal N_\varepsilon (W'_1)$ which govern the size of $V$ are
\[
\varepsilon^2 \, K_{g}(s) \, \left(\frac{t-\varepsilon\, \Psi_\varepsilon (s)}{\varepsilon} \right) \, W_{2}^{\prime\prime}\left(\frac{t-\varepsilon\, \Psi_\varepsilon (s)}{\varepsilon} \right)
+ \varepsilon^{2} \, \left(
\partial_{s}\Psi_\varepsilon (s)\right)^{2}\, W_{2}^{\prime\prime} \left(\frac{t-\varepsilon\, \Psi_\varepsilon (s)}{\varepsilon} \right).
\]
To evaluate their contribution to the size of $V$, we use Proposition~\ref{pr:4.3} which provides estimates for $\partial_s\Psi$ and we get most $C \, \varepsilon^2 (\ln \varepsilon)^2$.

\medskip

In $3\left(1-W_2^{2}\right)\left(1-W_1\right)  -3\left( 1+W_2\right)  \left(1-W_1\right)^{2}$, the term which governs the size of $V$ is
\[
3 \left(1-W_2^{2}\right)\left(1-W_1\right),
\]
whose contribution to the size of $V$ can be estimated by $C \, e^{-2 \alpha_\varepsilon}$ since $\sqrt 2 \, (f_2-f_1)$ is bounded from below by $2 \, \alpha_\varepsilon - C$.

\medskip

The existence of $V$, $h_1$ and $h_2$ follows from a fixed point argument. Beside the analysis of  which is given in \cite{M0}, invertibility of  $\mathfrak J_\gamma$ is needed and we also need to use Proposition~\ref{pr:5.4}.
\end{proof}

In the next section, we will study the Morse index of the solution we have just constructed. To this aim, we will need some better understanding of the function $V_\varepsilon$. This understanding is based on the following technical lemma part of which has been already used~:
\begin{lemma}
\label{le:6.1}
Let $X, Y$ and $Z$ be the solutions of
\[
\begin{array}{rllll}
\left(\partial_x^2 + 1 - 3 \, W^2 \right) \, X  & = & W'' \\[3mm]
\left(\partial_x^2 + 1 - 3 \, W^2 \right) \, Y &  = & x \, W' \\[3mm]
\left(\partial_x^2 + 1 - 3 \, W^2 \right) \, Z  & = &  6 \, (1-W^2) \, \left( e^{-\sqrt 2 \, x} - \rho \right),
\end{array}
\]
which are $L^2(\R)$ orthogonal to $W'$,
where the constant $\rho$ is determined so that the right-hand-side of the third equation is $L^2(\R)$ orthogonal to $W'$. Then
\[
\begin{array}{rlll}
\displaystyle 6 \, \int_\R W\, (W')^2 \, X \, dx & = & \displaystyle \int_\R (W'')^2 \, dx \\[3mm]
\displaystyle  6 \, \int_\R W\, (W')^2 \, Y\, dx & =- &\displaystyle \frac{1}{2}\, \int_\R (W')^2 \, dx \\[3mm]
\displaystyle 6 \, \int_\R  W\, (W')^2 \, Z \, dx & = & 6 \, \displaystyle \int_\R e^{-\sqrt 2 \, x} \, (1-W^2) \, W''\,  dx  .
\end{array}
\]
\end{lemma}
\begin{proof}
Assume that $A : \R \to \R$ is a solution of
\[
\left(\partial_x^2 + 1 - 3 \, W^2 \right)\, A = B,
\]
where $B : \R \to \R$ is a smooth function. Taking the derivative with respect to $x$, we get
\[
\left(\partial_x^2 + 1 - 3 \, W^2 \right)\, \partial_x A - 6 \, W\, W' \, A = \partial_x B,
\]
which we multiply by $W'$ and integrate by parts over $\R$ to conclude that
\[
6 \, \int_\R W \, (W')^2 \, A \, dx =  \int_\R B \, W'' \, dx.
\]
The result then follows easily.
\end{proof}

As a consequence we get the:
\begin{corollary}
\label{co:6.3}
The following estimate holds
\[
\begin{array}{rllll}
\displaystyle 6 \, \int_{(-c,c)}  W_2\, V \, (W_2')^2 \, dt & = & \displaystyle  - \varepsilon^3 \,(\partial_s f_2)^2 \, \int_\R (W'')^2\, dx + \varepsilon^3 \, K_g(s) \, \frac{1}{2}\, \int_\R (W')^2 \, dx \\[3mm]
&  & \displaystyle - 6 \, \varepsilon \, \, e^{-\sqrt 2 \,  (f_2-f_1)}\, \int_\R e^{-\sqrt 2\, x} \, (1-W^2) \, W'' \, dx \\[3mm]
&  & + \mathcal O (  \varepsilon^5 \, (\ln \varepsilon)^4).
\end{array}
\]
\end{corollary}
\begin{proof}
The result follows directly from Lemma \ref{le:6.1} and the fact that the remainder terms are of the order $\varepsilon^5 \, (\ln \varepsilon)^4$.
\end{proof}
\section{The index of the solutions of the Allen-Cahn equation}
In this section, we study the Morse index of the solutions to the Allen-Cahn equation we have constructed. We will also prove that our solutions are non-degenerate, i.e. that they have nullity equal to $0$.

\begin{proposition}
\label{pr:7.1}
Assume that $\Phi$ is a non-degenerate Bouncing Jacobi Field with $n$ minimums. Then, for all $\varepsilon >0$ close enough to $0$, the solution of the Allen-Cahn equation which has been constructed in Proposition \ref{pr:6.2} has nullity $0$ and Morse Index equal to $n + \text{Ind} \, (\Phi ) + \text{Ind} \, (\gamma)$.
\end{proposition}

The proof of this Proposition is decomposed into two parts. First, we show that the linearized operator
\[
\mathfrak L_\varepsilon : = -\left( \varepsilon^2 \, \Delta_g + 1 - 3 \, u_\varepsilon^2 \right),
\]
has Morse index at least equal to $n+\text{Ind} (\Phi)+ \text{ind} (\gamma)$ and then we prove that it has nullity $0$ and Morse index at most equal to $n+\text{Ind} (\Phi)+ \text{ind} (\gamma)$.

\medskip

We define
\[
\mathfrak Q_\varepsilon (\phi) := \displaystyle \int_M \left(\varepsilon^2 \, |\nabla \phi|_g^2 - (1 - 3\, u_\varepsilon^2) \, \phi^2 \right) \, \text{dvol}_g.
\]
Observe that, close to $\gamma$, we have, in Fermi coordinates,
\[
|\nabla \phi|_g^2 = (\partial_t\phi)^2 + \left(1 + K_g(s) \, t^2 + \mathcal O (t^3)\right)\, (\partial_s\phi)^2,
\]
and
\[
\text{dvol}_g = \left(1 - K_g(s)\, \frac{t^2}{2} + \mathcal O (t^3)\right) \, ds \, dt.
\]

We start with the following technical Lemma:
\begin{lemma}
\label{le:7.2}
Assume that $k : \gamma \to \R$ is a smooth function and
\[
\phi (s,t) = \chi (t) \, \left(W'_2 (s,t) \pm W'_1(s,t) \right) \, k(s),
\]
which is extended by $0$ away from the support of $\chi$. Then
\[
\begin{array}{rlll}
\mathfrak Q_\varepsilon (\phi) & = & \displaystyle 2 \, \varepsilon \, \left(\int_\R (W')^2 \, dx \right) \, \left( \int_\gamma \varepsilon^2 \, \left( (\partial_s k)^2 - K_g \, k^2 \right) \, ds \right) \\[3mm]
& - & \displaystyle  (1\mp 1) \, 2 \, \varepsilon\,  \left(6 \, \int_\R e^{-\sqrt 2 \, x} \, (1-W^2)\, W' \, dx \right) \, \left( \int_\gamma  e^{-\sqrt{2} (f_2-f_1)} \, k^2 \, ds \right) \\[3mm]
& + & \mathcal O (\varepsilon^4 \, (-\ln \varepsilon)^3).
\end{array}
\]
\end{lemma}
\begin{proof}
The proof is a simple consequence of Lemma~\ref{le:7.4} and an integration by parts. Observe that the error is not as good as in this Lemma because the term $\mathcal O (t^3)$ in the expansion of the volume form.
\end{proof}

Thanks to this technical result, we are able to prove the:
\begin{lemma}
\label{le:7.3}
For all $\varepsilon$ small enough,
$\mathfrak L_\varepsilon$ has Morse index \emph{at least} equal to $n + \text{Ind}\, (\Phi)+ \text{Ind}\, (\gamma)$.
\end{lemma}
\begin{proof}
We consider $k$ to be one of the functions $e_j$ and $\tilde \eta$ which have been constructed in the proof of Lemma~\ref{le:5.2}.
Then, the functions
\[
\phi (s,t) := \chi (t) \, \left(W'_2 (s,t) - W'_1(s,t) \right) \, k(s).
\]
provide a $n + \text{Ind} (\Phi)$ dimensional space over which  the quadratic form $\mathfrak Q_\varepsilon$ is negative definite, provided $\varepsilon$ is small enough.
\medskip

Now, if we assume that $k$ is one of the nonzero eigenfunction of $\mathfrak J_\gamma$ associated to a negative eigenvalue. We set
\[
\psi (s,t) := \chi (t) \, \left(W'_2 (s,t)  + W'_1(s,t) \right) \, k(s).
\]
As above, we get that the quadratic form $\mathfrak Q_\varepsilon$ is negative on the vector space spanned by these functions and this provides a $\text{Ind} (\gamma)$ dimensional space over which the quadratic form $\mathfrak Q_\varepsilon$ is negative definite, provided $\varepsilon$ is small enough.

\medskip

Since the above functions are linearly independent and this implies that  the linearized operator $\mathfrak L_\varepsilon$ has Morse index  at least equal to $n+\text{Ind} (\Phi)+ \text{Ind} (\gamma)$.
\end{proof}

\medskip

We keep the notations we have used in the previous section and we consider an eigenfunction $\phi$ of
\[
\mathfrak L_\varepsilon := - \left( \varepsilon^{2} \, \Delta_g + 1 - 3 \, u_\varepsilon^2 \right),
\]
the linearized operator at $u_\varepsilon$, associated to an eigenvalue $\lambda \leq 0$. Hence
\[
\mathfrak L_\varepsilon \, \phi = \lambda \, \phi.
\]
We assume the $L^\infty$ norm of $\phi$ equals $1$. In a fixed, small tubular neighborhood of size $2c$ about $\gamma$, we decompose the eigenfunction $\phi$ as
\begin{equation}
\label{eq:7.1}
\phi (s,t) = k_2 (s) \, W_2^{\prime} (s,t) + k_1 (s) \, W_1^{\prime}(s,t) + \eta (s,t),
\end{equation}
with
\[
\int_{(-c,c)} \eta (s,t)  \, W_2^{\prime} (s,t)\, dt  = \int_{(-c,c)} \eta (s,t) \, W_1^{\prime} (s,t)\, dt ,
\]
to address the uniqueness question of the decomposition.

\medskip

Observe that,
\[
\varepsilon^2 \, \Delta_g \phi - (3\, u_\varepsilon^2-1 -\lambda) \, \phi =0.
\]
Since $\lambda \leq 0$ and since, for any given $\delta >0$, $u_\varepsilon > 1- 2\delta/3$ away from the transition layers, which are close to the graphs of $\varepsilon\, f_2$ and $\varepsilon\, f_1$, the potential $3\, u_\varepsilon^2-1 -\lambda \geq 2-2\delta$ we conclude that $\phi$ tends to $0$ exponentially fast, as a function of the distance to the transition layers, like
\[
(s,t) \mapsto \left( W'_2 (s,t) \right)^{\sqrt{1-\delta}} + \left( W'_1 (s,t)\right)^{\sqrt{1-\delta}}.
\]

To get more insight on the eigenfunction $\phi$, we project the equation $\mathfrak L_\varepsilon \phi =-\lambda \phi$ over $W_2'$ and $W_1'$ and over the orthogonal complement of these functions. We start with the projection over $W_2'$.
\begin{lemma}
\label{le:7.4}
The following estimates hold:
\[
\begin{array}{rlllll}
\displaystyle \int_{(-c,c)}  \mathfrak L_\varepsilon (k_2\, W_2')\, W_2' \, dt  =  \varepsilon \,  c_0 \, \left( \varepsilon^2 \, \mathfrak J_\gamma + \bar c^2 \, e^{- 2 \Psi_\varepsilon} \right) \, k_2 + \varepsilon^5\, (\ln \varepsilon)^4  \, L (k_2),
\end{array}
\]
and
\[
\begin{array}{rlllll}
\displaystyle \int_{(-c,c)}  \mathfrak L_\varepsilon (k_1\, W_1')\, W_1' \, dt  =  \varepsilon \,  c_0 \, \left( \varepsilon^2 \, \mathfrak J_\gamma + \bar c^2 \, e^{- 2 \Psi_\varepsilon} \right) \, k_1 + \varepsilon^5\, (\ln \varepsilon)^4  \, L (k_2),
\end{array}
\]
where $L$ are second order linear operators with bounded coefficients and where the constants $c_0$ and $c_1$ have been defined in (\ref{eq:c0c1}).
\end{lemma}
\begin{proof}
We compute, using the expansion of the Laplace-Beltrami  in Fermi coordinates (\ref{lap-Bel}), we get
\[
\begin{array}{rllll}
\mathfrak L_\varepsilon (k_2 W'_2) & = & - \varepsilon^2 \, \left( \partial_s^2 k_2 \, W'_2  - (2 \, \partial_s f_2 \, \partial_s k_2 \,  - \partial_s^2 f_2 \, k_2 )\, W_2'' + (\partial_sf_2)^2 \, k_2 \, W_2''' \right) \\[3mm]
& - & \varepsilon^2 \, \mathcal O (t^2) \left( \partial_s^2 k_2 \, W'_2  - (2 \, \partial_s f_2 \, \partial_s k_2 \,  - \partial_s^2 f_2 \, k_2 )\, W_2'' + (\partial_sf_2)^2 \, k_2 \, W_2''' \right) \\[3mm]
& - & \varepsilon^2 \,  \mathcal O(t^2) \, \left( \partial_sk_2 \, W_2' - \partial_s f_2 \, k_2 \, W_2'' \right)\  + k_2\, W_2''' \\[3mm]
& + & \varepsilon^2 \, \left( K_g + \mathcal O (t)\right)  \,  k_2\, t \, W_2''  + k_2 W_2' - 3 (W_2 - (W_1-1)+ V)^2 \, k_2\,  W_2' \, .
\end{array}
\]
Using the fact that $(W')'' + W' - 3 W^2 \, W'=0$, the above expression simplifies slightly. Then, we can multiply the above expression by $W'_2$ and integrate (by parts) the result over $(-c,c)$. With little work, we find
\[
\begin{array}{rllll}
\displaystyle \int_{(-c,c)} \mathfrak L_\varepsilon (k_2 W'_2) \, W_2'\, dt & = & \displaystyle - \varepsilon^3 \, \partial_s^2 k_2 \, \int_\R (W')^2 \, dx +  \varepsilon^3 \, (\partial_sf_2)^2 \, k_2 \, \int_\R (W'')^2 \, dx \\[3mm]
& - & \displaystyle \varepsilon^3 \, K_g \,  k_2\, \frac{1}{2} \int_\R (W')^2 \, dx +  6 \, k_2 \, \int_{(-c,c)} W_2 \, V \,  (W_2')^2\, dt \\[3mm]
& + & \displaystyle 12 \, \varepsilon \, e^{- \sqrt 2 \, (f_2-f_1)}\, k_2 \, \int_\R e^{- \sqrt 2\, x} \, W \,  (W')^2 \, dx  \\[3mm]
& + & \varepsilon^5\, (\ln \varepsilon)^4  \, L (k_2),
\end{array}
\]
where $L$ is a second order linear operator with bounded coefficients. To obtain this estimate, we have implicitly used the fact that
\[
e^{-\sqrt 2\,  (f_2-f_1)}  = \mathcal O (\varepsilon^2 \, (\ln \varepsilon)^2)
\]
which follows from the construction of $f_2$ and $f_1$.
Finally, one can use Corollary~\ref{co:6.3} together with the fact that
\[
\displaystyle \int_\R e^{-\sqrt 2 \, x} \, (1-W^2) \, W''\,  dx  \\[3mm]
= \displaystyle 2 \, \int_\R e^{-\sqrt 2 \, x} \, W\, (W')^2\,  dx + \sqrt 2 \, \int_\R e^{-\sqrt 2 \, x} \, (1-W^2)\, W' \, dx,
\]
to conclude that
\[
\begin{array}{rllll}
\displaystyle \int_{(-c,c)} \mathfrak L_\varepsilon (k_2 W'_2) \, W_2'\, dt & = & \displaystyle - \varepsilon^3 \, \partial_s^2 k_2 \, \int_\R (W')^2 \, dx  - \displaystyle \varepsilon^3 \, K_g \,  k_2\, \frac{1}{2} \int_\R (W')^2 \, dx  \\[3mm]
& + & \displaystyle 6 \, \varepsilon \, e^{- \sqrt 2 \, (f_2-f_1)}\, k_2 \, \int_\R e^{- \sqrt 2\, x} \, (1-W^2) \, W' \, dx  \\[3mm]
& + & \varepsilon^5\, (\ln \varepsilon)^4  \, L (k_2),
\end{array}
\]
The result follows from the definition of $\bar c^2$ and the fact that
\[
\sqrt 2\, (f_2- f_1) = 2\, \Psi_\varepsilon + \mathcal O (\varepsilon^2 \, (\ln \varepsilon)^2).
\]
This completes the proof of the Lemma.
\end{proof}

Our next task is to perform a similar analysis for the projection of $\mathfrak L_\varepsilon \, (k_1\, W_1')$ over $W'_2$.
\begin{lemma}
\label{pro2}
The following estimates hold:
\[
\displaystyle \int_{(-c,c)}  \mathfrak L_\varepsilon (k_1\, W_1') \, W_2' \, dt  = - \varepsilon \,  c_0  \, \bar c^2 \, e^{- 2 \Psi_\varepsilon} \, k_1  + \varepsilon^5 \, (\ln \varepsilon)^4 \, L (k_1, k_2),
\]
and
\[
\displaystyle \int_{(-c,c)}  \mathfrak L_\varepsilon (k_2\, W_2') \, W_1' \, dt  = \varepsilon \,  c_0  \, \bar c^2 \, e^{- 2 \Psi_\varepsilon} \, k_2  + \varepsilon^5 \, (\ln \varepsilon)^4 \, L (k_1, k_2),
\]
where $L$ are second order linear operators with bounded coefficients.
\end{lemma}
\begin{proof}
The structure of $\mathfrak L_\varepsilon (k_1\, W_1')$ is similar to the structure we have obtained  in the previous Lemma for $\mathfrak L_\varepsilon (k_2\, W_2')$, exchanging the indices $1$ and $2$. As above we multiply the expression of $\mathfrak L_\varepsilon (k_1\, W_1')$ by $W'_2$ and integrate the result over $(-c,c)$. Using the fact that $(W'')'+ W'- 3\, W^2 \, W'=0$, we  get
\[
\begin{array}{rllll}
\displaystyle \int_{(-c,c)} \mathfrak L_\varepsilon (k_1 W'_1) \, W_2'\, dt & = & \displaystyle  3 \, k_1 \, \int_{(-c,c)}  \left( \left( W_1 - (1+W_2) -V \right)^2 - W_1^2\right) \, W'_1\, W_2'\, dt \\[4mm]
& + & \varepsilon^5\, (\ln \varepsilon)^4  \, L (k_1, \partial_s k_1, \partial_s^2 k_1),\\[4mm]
& = & \displaystyle  3 \, k_1 \, \int_{(-c,c)} (1-W_2^2) \, W'_1 \, W_2'\, dt  \\[4mm]
& + & \varepsilon^5\, (\ln \varepsilon)^4  \, L (k_1, \partial_s k_1, \partial_s^2 k_1),
\end{array}
\]
where $L$ changes from one line to the next but still has bounded coefficients, and we conclude that
\[
\begin{array}{rllll}
\displaystyle \int_{(-c,c)} \mathfrak L_\varepsilon (k_1 W'_1) \, W_2'\, dt & = &
\displaystyle  - \varepsilon \, k_1 \, e^{-\sqrt 2\, (f_2-f_1)}\, 6 \, \int_{\R} e^{-\sqrt 2\, x} (1 - W^2) \, W \, W' \, dx  \\[4mm]
& + & \varepsilon^5\, (\ln \varepsilon)^4  \, L (k_1, \partial_s k_1, \partial_s^2 k_1).
\end{array}
\]

A similar analysis can be done for the projection of $\mathfrak L_\varepsilon (k_2\, W'_2)$ over $W'_1$. The computation is exactly the same but one need to be careful that, this time the signs in front of
\[
e^{-\sqrt 2\, (f_2-f_1)}\, 6 \, \int_{\R} e^{-\sqrt 2\, x} (1 - W^2) \, W \, W' \, dx ,
\]
are changed.
\end{proof}

The next result shows that, in some sense, eigenfunctions of $\mathfrak L_\varepsilon$ associated to non-positive eigenvalues are essentially of the form $k_1\, W'_1 + k_2\, W_2'$.
\begin{lemma}
\label{le:7.6}
There exists $\alpha>0$ independend of $\varepsilon$ such that
\[
\| \eta\|_{\mathcal C^{2, \alpha}_\varepsilon} \leq C\, \varepsilon^4 \, (\ln \varepsilon)^4\, \left(
\left\Vert k_{1}\right\Vert_{\mathcal C^{2, \alpha}} +\left\Vert k_{2}\right\Vert_{\mathcal C^{2, \alpha}} \right)  .
\]
\begin{equation}
\label{etaor}
\begin{array}{rllll}
\displaystyle \int_{(-c,c)} \mathfrak L_\varepsilon \eta \, W_2'\, dt&  =&  C \, \varepsilon^4\, (\ln \varepsilon)^4 \, L (k_1, k_2), \\[3mm]
\displaystyle \int_{(-c,c)} \mathfrak L_\varepsilon \eta \, W_1'\, dt & = & C \, \varepsilon^4\, (\ln \varepsilon)^4  \,L (k_1, k_2),
\end{array}
\end{equation}
where $L$ are second order linear operators with bounded coefficients.
\end{lemma}
\begin{proof}
Using the expansion of Laplacian-Beltrami operator and the orthogonality assumption on $\eta$, we obtain
\begin{equation}
\label{etae}
\int_{(-c,c)}\left(  \mathfrak L_{\varepsilon
}\eta-\lambda\eta\right)\,   W_{2}^{\prime} \, dt =  \varepsilon^3 \, L\left(  \eta\right) .
\end{equation}

We set
\[
P := (\mathfrak L_\varepsilon  - \lambda ) \eta = -(\mathfrak L_\varepsilon  -\lambda )\, (k_2\, W'_2 + k_1 \, W'_1).
\]
We then write
\[
\displaystyle \mathfrak L_{\varepsilon}\eta-\lambda\eta=P-\frac{\displaystyle \int_{(-c,c)}P \, W_{2}^{\prime}\, dt}{\displaystyle \int_{(-c,c)} \left(  W_{2}^{\prime}\right)^{2} \, dt} \, W_{2}^{\prime}+\frac{\displaystyle \int_{(-c,c)}PW_{2}^{\prime}\, dt}{\displaystyle \int_{(-c,c)}\left( W_{2}^{\prime}\right)^{2}\, dt} \, W_{2}^{\prime}.
\]
Applying Lemma \ref{le:7.4} and Lemma \ref{pro2}, we get
\[
\displaystyle P-\frac{\displaystyle \int_{(-c,c)} P \, W_{2}^{\prime}\, dt}{\displaystyle \int_{(-c,c)}\left(W_{2}^{\prime}\right)^{2}\, dt} \, W_{2}^{\prime}= \varepsilon^4 \, (\ln \varepsilon)^4  \, L\left(k_{1},k_{2}\right).
\]
Hence we obtain%
\[
(\mathfrak L_{\varepsilon} -\lambda) \, \eta = \mathcal \varepsilon^4\, (\ln \varepsilon)^4 \, L\left( k_{1},k_{2}\right) + \varepsilon \, L\left(
\eta\right)  .
\]
From this we then deduce
\[
\| \eta\|_{\mathcal C^{2, \alpha}_\varepsilon} \leq C\, \varepsilon^4\, (\ln \varepsilon)^4 \,\left(
\left\Vert k_{1}\right\Vert_{\mathcal C^{2, \alpha}} +\left\Vert k_{2}\right\Vert_{\mathcal C^{2, \alpha}} \right)  .
\]
The estimate (\ref{etaor}) then follows from (\ref{etae}).
\end{proof}

Taking these Lemma into account, we find that $k_{1}$ and $k_{2}$ satisfy the system
\[
\left\{
\begin{array}{rllll}
\varepsilon^2 \, \mathfrak J_\gamma k_{1} + 2 \, \bar c^2\,  e^{-2 \Psi_\varepsilon}\,  \left(
k_{1}-k_{2}\right) = \lambda \, k_{1} + \varepsilon^4\, (\ln \varepsilon)^4 \, L (k_1,k_2),\\[3mm]
\varepsilon^2 \, \mathfrak J_\gamma k_{2} - 2 \, \bar c^2 \, e^{-2\Psi_\varepsilon} \, \left(
k_{1}-k_{2}\right) = \lambda \, k_{2} + \varepsilon^4\, (\ln \varepsilon)^4 \, L (k_1,k_2).
\end{array}
\right.
\]
where $L$ are second order linear operators with bounded coefficients.

\medskip

If we define $\xi :=k_{1}-k_{2}$ and $\zeta := k_{1} +k_{2}$, we see that these functions satisfy the system
\begin{equation}
\label{eq:7.2}
\left\{
\begin{array}{rllll}
\mathcal L_\varepsilon \xi & = & \lambda \, \xi + \varepsilon^4 \, (\ln \varepsilon)^4 \, L(\xi , \zeta)\\[3mm]
\varepsilon^2 \, \mathfrak J_\gamma \zeta  & = & \lambda \,  \zeta +  \varepsilon^4 \, (\ln \varepsilon)^4 \, L(\xi , \zeta).
\end{array}
\right.
\end{equation}

We are now in a position to prove the:
\begin{lemma}
\label{le:7.7}
For all $\varepsilon$ small enough the linearized operator $\mathfrak L_\varepsilon$ has nullity $0$ and Morse index at most equal to $n+ \text{Ind}\, (\Phi) + \text{Ind}\, (\gamma)$.
\end{lemma}
\begin{proof}
The analysis we have done in the proof of Lemma~\ref{le:5.3} can be duplicated and we conclude that, for all $\varepsilon$ small enough, the index of $\mathfrak L_\varepsilon$ is \emph{at most} equal to
\[
n + \text{\rm Ind}\, (\gamma)+\text \, {\rm Ind} \, (\Phi),
\]
and that $\mathfrak L_\varepsilon$ does not have any kernel. We leave the details to the reader.
\end{proof}

To summarize, we have proven that the linearized operator $\mathfrak L_\varepsilon$ does not have any kernel and  we have also proven that its Morse index is at least and at most equal to $n+\text{Ind} (\Phi)+ \text{Ind} (\gamma)$. Therefore, the proof of Proposition~\ref{pr:7.1} is complete.

\medskip

Theorem \ref{th:1.5} then follows directly from Proposition \ref{pr:6.2} and Proposition \ref{pr:7.1}.
\section{Bumpy metrics}
In this section, we prove that both the geodesics we work with and the corresponding Bouncing Jacobi Fields can be assumed to be non-degenerate, for generic metrics in the sense of Baire category. This follows from application of the Sard-Smale's Theorem and more precisely of the Transversality Theorem 2 in \cite{Uhl} and hence, we will adopt in this section most notations of \cite{Uhl}.

\medskip

We fix $k\geq 3$ and let $\mathcal M$ be the set of $\mathcal C^k$ metrics on the surface $M$. Let $\Gamma$ denote the space of $\mathcal C^2$ embedding $\gamma : S^1 \to M$. We define $\mathcal K$ to be the set of couples $(g,[\gamma]) \in \mathcal M \times \Gamma$ such that the image of $\gamma$ is an embedded closed geodesic of $(M,g)$ and where $[\gamma]$ denotes the set of $\gamma\circ \phi$ where $\phi : S^1 \to S^1$ is a diffeomorphism. We also define the projection
\[
\pi : \mathcal K \to \mathcal M,
\]
by $\pi (g,\gamma)=g$. It is proven in \cite{Whi91} that $\mathcal K$ is a separable Banach manifold and $\pi$ is a Fredholm map of index $0$. Moreover, of $g$ is a regular value of $\pi$ then all closed embedded geodesics of $(M,g)$ are non-degenerate.

\medskip

We set
\[
X \subset T^n = S^1 \times \ldots \times S^1,
\]
to be the set of ordered $n$-tuples of distinct points on $S^1$ and we define $F : \mathcal M \times X \to \R^n$ by
\[
F((g, [\gamma]), \Theta ) := \left( \frac{\partial_\theta \Phi(\theta_j^+) + \partial_\theta \Phi (\theta_j^-)}{|\partial_\theta \gamma (\theta_j)|_g}\right)_{j=1, \ldots, n},
\]
where $\Theta :=(\theta_1, \ldots, \theta_n)$, $\Phi$ solves
\[
\frac{1}{|\partial_\theta \gamma|_g} \partial_\theta \left(\frac{1}{|\partial_\theta \gamma|_g} \partial_\theta \Phi\right) +K_g (\gamma) \, \Phi =0 \qquad \text{and} \qquad \Phi >1,
\]
in each $(\theta_j, \theta_{j+1})$ and $\Phi (\theta_j) =1$.

\medskip

Observe that Bouncing Jacobi Fields are in one to one correspondence with  solutions of $F(g,\gamma, \Theta)=0$. We start with the:
\begin{lemma}
\label{le:8.1}
The Bouncing Jacobi Field associated to $(g, [\gamma]) \in \mathcal K \times \Gamma$ is non-degenerate if $g$ is a regular value of $\pi$ and if $0$ is a regular value of $F _{|\pi^{-1}(g)\times X}$.  \end{lemma}
\begin{proof}
The fact that $0$ is a regular value of $\pi$ implies that all geodesics are non-degenerate and hence isolated. The fact that $0$ is a regular value of the restriction of the function $F$ to $\pi^{-1}(g)\times X$ implies that the corresponding Bouncing Jacobi Fields are non-degenerate.
\end{proof}

We now prove the:
\begin{lemma}
\label{le:8.2}
Assume that $0$ is a regular value of $F : \mathcal K \times X\to \R^n$. Then, the set of $g \in \Gamma$ such that the embedded closed geodesics are non-degenerate and carry a non-degenerate Bouncing Jacobi Field is residual in the space of $\mathcal C^k$ metrics, for $k\geq 3$.
\end{lemma}
\begin{proof}
This result is a consequence of the Transversality Theorem in \cite{Uhl}.
\end{proof}

In view of the previous Lemma, there remains to prove that $0$ is a regular value of $F : \mathcal K \times X \to \R^n$. To this aim, we start with the expression of the first variation of $F$ with respect to $g$. In Fermi coordinates, the metric $g$ can we written as
\[
g = (1 + K_g(s) \,  t^2 + \mathcal O (t^3)) \, ds^2 + dt^2,
\]
and we consider an infinite dimensional family of metrics defined in a tubular neighborhood of the geodesic by
\[
g_z =  g + z(s) \, t^2 \, ds^2 .
\]
where $s \mapsto z(s)$ is a small smooth function.

\medskip

Observe that with this choice the curve $\gamma$ is a geodesic of $M$ endowed with $g_z$ since all partial first order derivatives of the coefficients  of the metric $g_z$ vanish at $t=0$ and hence the Christoffel symbols also vanish at $t=0$. This immediately implies that the curve $\gamma$, which is a geodesic when the manifold is endowed with the metric $g$, remains a geodesic when we change the metric into $g_z$. Let us write
\[
g_z=E_z \, ds^2 + G_z \, dt^2.
\] According to Brioschi's formula, the Gauss curvature $K_z$ of the perturbed metric $g_z$ is given by
\[
K_z =-\frac{1}{2\sqrt{E_zG_z}}\left(\partial_s\left(\frac{\partial_s G_z}{\sqrt{E_zG_z}}\right)+\partial_t\left(\frac{\partial_tE_z}{\sqrt{E_zG_z}}\right)\right),
\]
along $\gamma$. In our case $G_z \equiv 1$ and hence, this formula simplifies into
\begin{equation}
\label{eq:8.1}
    K_z = -\frac{1}{\sqrt{E_z}}\partial^2_t(\sqrt{E_z}).
\end{equation}
\medskip
Starting from a Bouncing Jacobi Field $\Phi$, we define $\Phi_z$ to be the solution of
\[
(\partial_s^2 + K_z)\, \Phi_z = 0,
\]
on each $(s_j, s_{j+1})$ with $\Phi_z(s_j) =1$, for $i=1, \ldots, n$.

\medskip

Taking the differential with respect to the function $z$ at $z=0$, we get that
\[
\dot \Phi := D_z \Phi_{z|z=0}(\dot z),
\]
is a solution of
\[
(\partial_s^2 + K_g)\, \dot \Phi + D_z K_{z|z=0}(\dot z) \, \Phi = 0,
\]
on each $(s_j, s_{j+1})$ with $\dot \Phi (s_j) =0$, for $i=1, \ldots, n$.

\medskip

Multiplying by $\Phi$ and integrating the result over $(s_j, s_{j+1})$ we conclude that
\[
\partial_s \dot \Phi (s_{j+1}^-) - \partial_s \dot \Phi (s_j^+) =- \int_{s_j}^{s_{j+1}} D_z K_{z|z=0}(\dot z) \, \Phi^2\, ds,
\]
for all $j=1, \ldots, n$.
Using (\ref{eq:8.1}), we get with little work
\[
D_z K_{z|z=0}(\dot z) = - \dot z,
\]
along $\gamma$.
 To conclude positively, given $\xi := (\xi_1, \ldots, \xi_n) \in \R^n$, we need to find $s \mapsto \dot z(s)$ such that
\[
\sum_{j=1}^n \left(  \partial_s \dot \Phi (s_j^+) + \partial_s \dot \Phi (s_j^-) \right)\, \xi_j \neq 0.
\]
Now, it is  a simple exercise to check that it is  possible to choose $\dot z$ such that
\[
\partial_s \dot \Phi (s_j^+) = 0 \qquad \text{and} \qquad  \partial_s \dot \Phi (s_j^-) =\xi_j,
\]
for all $j=1, \ldots, n$. This completes the proof of the fact that $0$ is a regular value of $F$ and hence completes the proof of Theorem \ref{th:1.7}:

\begin{theorem}
\label{th:8.3}
The set of metrics $g$ on $M$ such that closed embedded geodesics are non-degenerate and its associated Bouncing Jacobi Fields are non-degenerate is residual and generic in the space of $\mathcal C^k$ metrics, for $k\geq 3$.
\end{theorem}

\end{document}